\documentclass[10pt,a4paper]{amsart}
\usepackage[a4paper,margin=1in]{geometry}

\usepackage{euler}
\usepackage{bm}
\usepackage{amsmath,amsfonts,amssymb,amsthm}
\usepackage{amscd,tikz-cd,tikz,mathrsfs}
\usepackage{stmaryrd}
\usepackage{svg}
\usepackage{caption}
\usepackage{xcolor}
\usepackage{enumerate}
\usepackage[utf8]{inputenc}
\usepackage[english]{babel}
\usepackage[
backend=biber,
style=alphabetic,
sorting=nyt
]{biblatex}
\usepackage[colorlinks = true,
            linkcolor = blue,
            urlcolor  = blue,
            citecolor = blue,
            anchorcolor = blue]{hyperref}

\DeclareUnicodeCharacter{0301}{\'{e}}

\usetikzlibrary{trees}
\usetikzlibrary{decorations.markings}
\usetikzlibrary{arrows}

\theoremstyle{plain}
   \newtheorem{theorem}{Theorem}[subsection]
   \newtheorem{proposition}[theorem]{Proposition}

   \newtheorem{corollary}[theorem]{Corollary}
    \newtheorem{lemma}[theorem]{Lemma}
\theoremstyle{definition}
    \newtheorem{definition}[theorem]{Definition}
    \newtheorem{conjecture}[theorem]{Conjecture}
    \newtheorem{question}[theorem]{Question}
    
\newtheorem{example}[theorem]{Example}

\theoremstyle{remark}
\newtheorem{remark}{Remark}
\newcommand{\ZZ}{\mathbb{Z}}
\newcommand{\RR}{\mathbb{R}}

\newcommand{\PP}{\mathbb{P}}
\newcommand{\TT}{\mathbb{T}}

\newcommand{\EE}{\mathcal{E}}
\newcommand{\FF}{\mathcal{F}}
\newcommand{\LL}{\mathcal{L}}
\newcommand{\MM}{\mathcal{M}}
\newcommand{\KK}{\mathcal{K}}
\newcommand{\OO}{\mathcal{O}}
\newcommand{\DD}{\mathcal{D}}
\newcommand{\QQ}{\mathcal{Q}}

\newcommand{\tp}[1]{{{\mathbb{T}\mathbb{P}}^{#1}}}
\newcommand{\TP}{{\mathbb{T}\mathbb{P}}}
\newcommand{\trop}{\textsf{trop}}
\newcommand{\tropdet}{\text{tropdet}}
\newcommand{\troprank}{\text{troprank}}
\newcommand{\G}{\text{G}}
\newcommand{\Mat}{\text{Mat}}
\newcommand{\GL}{\text{GL}}

\newcommand{\Div}{\text{Div}}
\newcommand{\Pic}{\text{Pic}}
\newcommand{\Jac}{\text{Jac}}
\newcommand{\Gr}{\text{Gr}}
\newcommand{\Hom}{\textsf{Hom}}
\newcommand{\rank}{\text{rank}}

\newcommand{\id}{\text{id}}
\newcommand{\im}{\text{im}}

\renewcommand{\div}{\text{div}}

\renewcommand{\phi}{\varphi}

\graphicspath{{figures/}}

\begin{document}

\title{A tropical framework for using Porteous formula}
\author[A.~R.~Tawfeek]{Andrew R. Tawfeek}
\email[A.~R.~Tawfeek]{atawfeek@uw.edu}

\begin{abstract}
Given a rational polyhedral space $X$ (a tropical cycle with boundary, in the sense of Mikhalkin--Rau), one can define tropical vector bundles on $X$ having real or tropical fibers. By restricting attention to bounded rational sections of these bundles, one obtains characteristic classes that behave as expected classically. We develop further properties of these classes and use them to prove a tropical analogue of the splitting principle, which allows us to establish the foundations for Porteous' formula in this setting: a determinantal expression for the fundamental class of degeneracy loci in terms of Chern classes. The boundary framework is essential, as it allows the rank of a bundle morphism to drop at sedentary strata, giving degeneracy loci their expected codimension.
\end{abstract}

\maketitle

\section{Introduction}

In classical algebraic geometry, Porteous' formula describes the \textit{fundamental class} of degeneracy loci: the locus of points where the rank of a bundle morphism $\phi: \mathcal{E} \to \mathcal{F}$ over $X$ falls below an integer $k$.

\begin{theorem}[Porteous Formula]
Let $\phi: \mathcal{E} \to \mathcal{F}$ be a morphism of vector bundles of ranks $e$ and $f$ on a smooth variety $X$. If the scheme $D_k(\phi) \subseteq X$ has codimension $(e-k)(f-k)$, then its class is given by
$$[D_k(\phi)] = \Delta_{f-k}^{e-k} \left[ \frac{c(\mathcal{F})}{c(\mathcal{E})} \right].$$
\end{theorem}

The right-hand side of the Porteous formula is a Sylvester determinant whose entries depend only on the Chern classes of $\mathcal{F}$ and $\mathcal{E}$. In other words, by expressing a subspace of interest as a degeneracy locus of a bundle morphism, one can deduce topological information from knowledge of the bundles alone.

We adapt this statement to the tropical setting using the language of tropical vector bundles and Chern classes, largely building on the work of \cite{chern}. We prove a tropical analogue of Porteous' formula in the rank-$0$ case (Theorem \ref{main-thm}), with the resulting fundamental class residing in the tropical Chow ring:

\begin{theorem}[Tropical Porteous Formula in Rank $0$]
Consider a morphism of tropical vector bundles $\phi: \mathcal{E} \to \mathcal{F}$ with bounded matrix entries over a rational polyhedral space $X$, where the bundles have ranks $e$ and $f$, respectively. Suppose $D_0(\phi)$ has the expected codimension $ef$. Then
$$[D_0^\phi (\mathcal{E}, \mathcal{F})] = \Delta_f^e \left( \frac{c_{[t]}(\mathcal{F})}{c_{[t]}(\mathcal{E})} \right),$$
where $c_{[t]}(-)$ denotes the Chern polynomial of the vector bundle.
\end{theorem}

Along the way, we develop the necessary constructions for tropical vector bundles, including Chern classes, projectivization, intersection products, and a tropical splitting principle.

\subsection{Acknowledgments}
The author is grateful to Farbod Shokrieh for introducing him to this problem and for his mentorship. He also thanks Tuomas Tajakka and Siddharth Mathur for helpful conversations regarding the splitting lemma, Leopold Mayer regarding decomposition of bundles over $\TP^1$, and Nathan Pflueger for discussions of applications of a tropical Porteous formula. The author thanks the anonymous referee for a thorough and insightful report that motivated the adoption of the rational polyhedral space framework and the Hom bundle argument in the proof of the main theorem.

\section{Tropical Preliminaries}

We recall the foundational terminology of polyhedral complexes, following the framework of rational polyhedral spaces developed in \cite{main}. This framework naturally incorporates \textit{boundary} (via the extended tropical affine space $\TT^n$), which is essential for defining degeneracy loci of the expected codimension. The reader unfamiliar with these notions may treat them in analogy to their classical algebraic geometric counterparts; as we will see, the manner in which we interact with them --- for instance in manipulations with Chern classes --- is almost indistinguishable from the classical theory.

An equivalent approach to tropical vector bundles uses \textit{integral affine manifolds} and a more sheaf-theoretic language; see \cite{shaw} or \cite{gross2023semi}.

\subsection{Rational Polyhedral Spaces}

We take the polytopal perspective following \cite{main, polycomplexes}. The key departure from the setup of \cite{polycomplexes} is that we work in the \textit{extended tropical affine space} $\TT^n = (\RR \cup \{-\infty\})^n$ rather than $\RR^n$. This allows polyhedra to have faces meeting the boundary of $\TT^n$, which is the mechanism by which morphism rank can drop.

\begin{definition}[Sedentarity]
For a point $p \in \TT^n$, the \textit{sedentarity} of $p$ is the set $I(p) = \{i \in [n] \mid p_i = -\infty\}$. The \textit{sedentarity order} $\text{sed}(p) = |I(p)|$ counts the number of coordinates equal to $-\infty$. For each subset $I \subseteq [n]$, the \textit{sedentary stratum} $\TT^n_I = \{p \in \TT^n \mid I(p) = I\}$ is naturally identified with $\RR^{n - |I|}$, and we have
$$\TT^n = \bigsqcup_{I \subseteq [n]} \TT^n_I.$$
The open stratum $\TT^n_\emptyset = \RR^n$ is called the \textit{interior}.
\end{definition}

\begin{definition}[Rational polyhedral complex in $\TT^n$]
    A \textit{rational polyhedral complex} $X$ in $\TT^n$ is a weighted, pure-dimensional, rational, finite polyhedral complex in $\TT^n$ (with underlying lattice $\ZZ^n \subseteq \RR^n$) that satisfies the \textit{balancing condition} for every $\tau \in X^{(k-1)}$:
    $$\sum_{\sigma : \tau < \sigma} \omega(\sigma) v_{\sigma/\tau} = 0$$
    where $v_{\sigma/\tau}$ denotes the primitive integer generator of the ray obtained from projecting $\sigma$ to the vector quotient $\RR^n/V_\tau$, with $V_\tau$ denoting the linear vector space spanned by $\tau$. Faces of $X$ may be \textit{sedentary}, meaning they are contained in a boundary stratum $\TT^n_I$ for some $I \neq \emptyset$. When a ridge $\tau$ is sedentary with $\tau \subseteq \TT^n_I$, the balancing condition is formulated within the stratum $\RR^{n-|I|}$ in which $\tau$ has its interior; see \cite{main} Ch.~4 for details.
\end{definition}

We call the top-dimensional polyhedra in a pure-dimensional $X$ the \textit{facets} and the codimension one polyhedra the \textit{corners} (or \textit{ridges}) of $X$. We denote by $|X|$ the \textit{support}, which is the union of all facets of $X$ having non-zero weight. Two rational polyhedral complexes are \textit{equivalent} if they admit a common refinement with the same induced weights. A \textit{rational polyhedral space} (or \textit{tropical cycle}) $X$ is an equivalence class of rational polyhedral complexes. When $X$ has no sedentary faces (i.e., $|X| \subseteq \RR^n$), this recovers the framework of \cite{polycomplexes}.

\begin{remark}
The framework of rational polyhedral spaces in $\TT^n$ was developed by Mikhalkin--Rau \cite{main}. In this setting, the boundary strata of $\TT^n$ give the ambient space the structure of a manifold with corners. When a polyhedral complex $X$ has faces meeting these boundary strata, continuous functions on $|X|$ (including matrix entries of a bundle morphism) can degenerate to $-\infty$ as one approaches the boundary, enabling rank to drop. This is precisely the mechanism by which degeneracy loci acquire positive codimension.
\end{remark}

\subsection{Tropical Intersection Theory}

The intersection-theoretic machinery of Allermann--Rau \cite{polycomplexes} extends to rational polyhedral spaces in $\TT^n$; see \cite{main} for the general development. We recall the key definitions below, noting that they are stated for rational polyhedral spaces (which may have sedentary faces) but specialize to the original formulations of \cite{polycomplexes} when restricted to $\RR^n$.

\begin{definition}[Abstract tropical cycles]
Let $((X,|X|),\omega_X)$ be an $n$-dimensional rational polyhedral complex. Its equivalence class $[((X,|X|), \omega_X)]$ is called an \textit{(abstract) tropical $n$-cycle}. The \textit{set of $n$-cycles} is denoted by $Z_n$. We additionally define $|((X,|X|), \omega_X)| := |X^*|$. Like in the affine case, an $n$-cycle $((X,|X|),\omega_X)$ is called an \textit{(abstract) tropical variety} if $\omega_X(\sigma) \geq 0$ for all $\sigma \in X^{(n)}$.
\end{definition}

As an illustration, consider the following figure depicting an abstract tropical polyhedral complex, which becomes a tropical variety when each cone $\sigma$ is assigned weight one.

\begin{figure}[h]
\centering
\includegraphics[width=.3\linewidth]{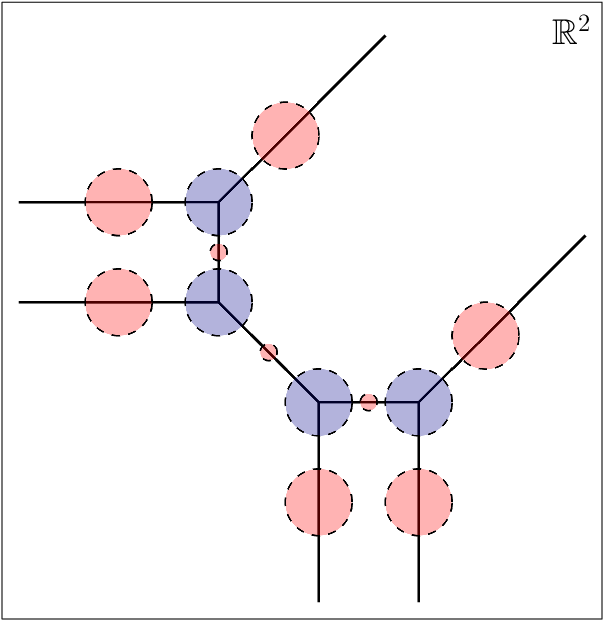} \quad \quad \quad 
\includegraphics[width=.15\linewidth]{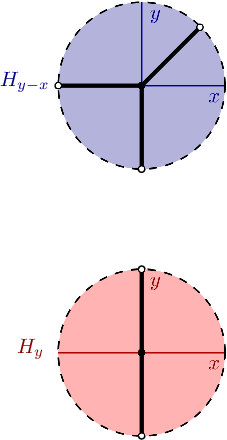}

\caption{An abstract tropical polyhedral complex, as well as a covering with open fans. Here, $H_f \subseteq \RR^2 $ denotes the fan having cones $\{ f(x) \leq 0\}$,  $\{f(x) = 0 \}$, and $\{f(x) \geq 0 \}$.}
\end{figure}

A tropical subcycle is a tropical cycle living within a larger ambient cycle, inheriting the expected structure from the ambient space.

\begin{definition}[Abstract tropical subcycles]
Let $C \in Z_n$ and $D \in Z_k$ be two tropical cycles. Then $D$ is called an \textit{(abstract) tropical cycle in $C$} or a \textit{subcycle of $C$} if there exists a representative $(((Z,|Z|), \omega_Z), \{\Psi_\tau\})$ of $D$ and a reduced representative $(((X,|X|), \omega_X), \{\Phi_\sigma\})$ of $C$ such that
\begin{enumerate}
	\item[(a)] $(Z,|Z|) \trianglelefteq (X,|X|)$,
	\item[(b)] the tropical structures on $Z$ and $X$ agree, i.e. for every $\tau \in Z$ the maps $\Psi_\tau \circ \Phi^{-1}_{C_{Z,X}(\tau)}$ and $\Phi_{C_{Z,X}(\tau)} \circ \Psi_\tau^{-1}$ are integer affine linear where defined.
\end{enumerate}
The \textit{set of tropical $k$-cycles} in $C$ is denoted by $Z_k(C)$.
\end{definition}

In order to define Cartier divisors, we need the analogous notions of rational and regular invertible functions.

\begin{definition}[Rational functions]
Let $C$ be an abstract $k$-cycle and let $U$ be an open set in $|C|$. A \textit{(non-zero) rational function on $U$} is a continuous function $\phi: U \to \RR$ such that there exists a representative $(((X,|X|,\{m_\sigma\}_{\sigma \in X}),\omega_X),\{M_\sigma\}_{\sigma \in X})$ of $C$ such that for each face $\sigma \in X$ the map $\phi \circ m_\sigma^{-1}$ is locally integer affine linear (where defined). The \textit{set of all non-zero rational functions on $U$} is denoted by $\mathcal{K}_C^*(U)$ (or $\mathcal{K}^*(U)$).
\end{definition}

\begin{definition}[Regular invertible functions]
With the same notation as above, if for each face $\sigma \in X$ the map $\phi \circ M_\sigma^{-1}$ is locally integer affine linear (where defined), $\phi$ is called \textit{regular invertible}. The \textit{set of all regular invertible functions on $U$} is denoted by $\mathcal{O}_C^*(U)$ (or $\mathcal{O}^*(U)$).
\end{definition}

\begin{remark}
It is not difficult to see that for an affine $k$-cycle $C$, we have that $(\mathcal{K}^*(C), \odot)$ is an abelian group and $(\OO^*(C), \odot)$ is a subgroup.
\end{remark}

With these ingredients in place, we can define divisors.

\begin{definition}[Cartier divisors] 
Let $C$ be an abstract $k$-cycle.  The set of \textit{Cartier divisors on $C$} is denoted by $\Div(C)$.
\begin{itemize}
	\item a \textit{representative of a Cartier divisor on $C$} is a finite set $\{(U_1,\phi_1), \dots, (U_n, \phi_n)\}$, where $\{U_i\}$ is an open covering of $|C|$ and $\phi_i \in \mathcal{K}^*(U_i)$ are rational functions on $U_i$ such that for $i \neq j$, we have that\footnote{This is equivalent to saying $$\frac{\phi_i|_{U_i \cap U_j}}{\phi_j|_{U_i \cap U_j}} = \phi_i|_{U_i \cap U_j} \odot (\phi_j|_{U_i \cap U_j})^{-1} = \phi_i|_{U_i \cap U_j} - \phi_j|_{U_i \cap U_j} \in \mathcal{O}^*(U_i \cap U_j).$$}
$$\frac{\phi_i|_{U_i \cap U_j}}{\phi_j|_{U_i \cap U_j}} \in \mathcal{O}^*(U_i \cap U_j).$$
	\item We define the \textit{product} of two representatives of Cartier divisors to be
$$\{(U_i, \phi_i)\} \odot \{(V_j, \psi_i)\} := \{(U_i \cap V_j, \phi_i \odot \psi_j)\}.$$
	\item We say two Cartier divisor representatives are \textit{equivalent} if
$$\frac{\{(U_i, \phi_i)\}}{\{(V_j, \psi_j)\}} = \{ (W_k, \gamma_k) \} \quad \text{ where } \gamma_k \in \mathcal{O}^*(W_k) \text{ for each $k$.}$$
\end{itemize}
It follows from the above constructions that the identity element of $\Div(C)$ is $[\{(|C|,0)\}]$, where $0$ denotes the constant zero function. We define the \textit{group of affine Cartier divisors of $C$} to be the quotient group
$\Div(C) := \KK^*(C)/\OO^*(C).$
\end{definition}

Lastly, the notion of a Weil divisor will be necessary for defining the intersection product, through which we later interact with higher Chern classes.

\begin{definition}[Associated Weil divisors]
Let $C$ be an affine $k$-cycle and $\phi \in \KK^*(C)$ a rational function on $C$. Let furthermore $(X,\omega)$ be a representative of $C$ on whose cones $\phi$ is affine linear. We define the \textit{Weil divisor} of $\phi$, denoted by $\div(\phi)$, to be
$$\varphi \cdot C := \left[ \left( \bigcup_{i=0}^{k-1} X^{(i)},\omega_\varphi \right) \right] \in Z_{k-1}^\text{aff}(C)$$
where the weight function $\omega_\varphi: X^{(k-1)} \to \ZZ$ is given by
$$\omega_\varphi(\tau) = \sum_{\begin{smallmatrix} \sigma \in X^{(k)} \\ \tau < \sigma \end{smallmatrix}} \varphi_\sigma(\omega(\sigma)v_{\sigma/\tau}) - \varphi_\tau\left( \sum_{\begin{smallmatrix} \sigma \in X^{(k)} \\ \tau < \sigma \end{smallmatrix}} \omega(\sigma) v_{\sigma/\tau} \right)$$
with the $v_{\sigma/\tau}$ denoting arbitrary representatives of the normal vectors $u_{\sigma/\tau} \in \Gamma_\sigma/\Gamma_\tau$. Furthermore, if $D$ is an arbitrary subcycle of $C$, then we define $\phi \cdot D := \phi|_{|D|} \cdot D$.
\end{definition}

\begin{definition}[Affine intersection product, \cite{polycomplexes}]
Let $C$ be an affine $k$-cycle and $\phi \in \Div(C)$ a Cartier divisor. The natural bilinear mapping\footnote{Where, recall, if $D$ is a subcycle of $C$, then $\phi \cdot D = \div(\phi|_{|D|})$.}
\[\begin{tikzcd}[row sep = tiny]
	{\text{Div}(C) \times Z_k^\text{aff}(C)} && {Z_{k-1}^\text{aff}(C)} \\
	{([\varphi],D)} && {\varphi\cdot D :=[\varphi] \cdot D }
	\arrow[from=1-1, to=1-3]
	\arrow[maps to, from=2-1, to=2-3]
\end{tikzcd}\]
is called the \textit{affine intersection product}.
\end{definition}

\section{Tropical Vector Bundles}

\subsection{Tropical Linear Algebra}

\begin{definition}[Tropical matrices]
A \textit{tropical matrix} will be a matrix with entries in the tropical semi-ring $(\TT, \oplus, \odot)$ where $\TT := \RR \cup \{-\infty\}$ and
$$a \oplus b := \max\{a,b\} \quad \text{ and } \quad a \odot b := a+b.$$
We will be using the following notation:
\begin{itemize}
	\item $\Mat(m\times n, \TT)$ for the set of tropical $m \times n$ matrices,
	\item $\G(r \times s) \subseteq \Mat(r\times s, \TT)$ for the subset  having at most one finite entry in each row,
	\item $\G(r) \subseteq \G(r \times r)$ for the subset with exactly one finite entry in every row and column.
\end{itemize}
Lastly, when multiplying matrices $A \in \Mat(m \times n, \TT)$ and $B \in \Mat(n \times p, \TT)$, we take their tropical product $A \odot B := (c_{ij}) \in \Mat(m \times p, \TT)$ where $c_{ij} = \bigoplus_{k=1}^n a_{ik} \odot b_{kj}$.
\end{definition}

\begin{remark} \label{fix}
A matrix $A \in \G(r \times s)$, for an $x \in \RR^s$, may send $x$ to $A \odot x \not \in \RR^r$ (i.e. contain entries that are $-\infty$). If we need to interpret such a matrix as a map  $f_A: \RR^s \to \RR^r$, we redefine any entries that are $-\infty$ in $A \odot x$ to be zero.
\end{remark}

For a permutation $\sigma \in S_n$, we denote by $E_\sigma \in \Mat(n \times n, \TT)$ the matrix
$$e_{ij} := \begin{cases} 0 &\text{if $j = \sigma(i)$} \\ -\infty &\text{otherwise,} \end{cases}$$
which plays a role analogous to the permutations matrices, but in the tropical setting. Furthermore, for choices of $a_i \in \RR$, we denote by $D(a_1,\dots, a_n) \in \G(n)$ the matrix with $a_1, a_2,\dots, a_n$ along the diagonal and $-\infty$ elsewhere.

\begin{remark} \label{help}
Using this, we may then write any matrix $M \in \G(n)$ as 
$$M = E_\sigma \odot D(a_1,\dots,a_n)$$
for some $\sigma \in S_n$ and $a_i \in \RR$. This makes it evident that $\G(n)$ is a group under tropical matrix multiplication, having identity element $D(0,\dots,0)$, which we denote by $E_n$ to distinguish it from $I_n \in \Mat(n \times n, \RR)$.
\end{remark}

\begin{lemma}
The group $\G(n) \subseteq \Mat(n \times n, \TT)$ is precisely the collection of invertible tropical matrices:
$$\GL_n(\TT) = G(n) = \{ A \in \Mat(n \times n, \TT) \ | \ (\exists B \in \Mat(n \times n, \TT)): \ A \odot B = B \odot A = E_n \}.$$
\end{lemma}

The proof is a short computation (resembling those in the proof of Lemma \ref{troprank-props} below), as these matrices behave similarly to permutation matrices. In fact, $\GL_n(\TT)$ is isomorphic to $S_n \ltimes \RR^n$, so one may think of invertible tropical matrices as permutation matrices with $0$'s replaced by $-\infty$ and $1$-entries replaced by elements of $\RR$.

\begin{definition}[Tropical determinant]
Consider a tropical matrix $A \in \Mat(n \times n, \TT)$. We define the \textit{tropical determinant} to be
$$\text{tropdet}(A) = \bigoplus_{\sigma \in S_n} (A_{1,\sigma(1)} \odot A_{2,\sigma(2)} \odot \cdots \odot A_{n, \sigma(n)})$$
\end{definition}

\begin{lemma} \label{permmatrixrep}
Consider a tropical matrix $A \in \G(n \times n)$. We then have that
$$\tropdet(A) \neq -\infty \iff A \in \G(n).$$
That is to say, the only tropical matrices with ``non-trivial" determinants are precisely those that are invertible.
\end{lemma}

\begin{proof}
The reverse direction is clear, so we focus on the forward direction. Recall that if $A \in \G(n\times n)$, then $A$ has \textit{at most} one finite entry in each row. Without loss of generality, assume the finite entries (if they exist) of $A$ lie along the diagonal, i.e.
$$A = \begin{pmatrix}
a_1 & -\infty & \cdots & -\infty \\
-\infty & a_2 & \cdots & -\infty \\
\vdots & \vdots & \ddots & \vdots \\
-\infty & -\infty & \cdots & a_n
\end{pmatrix}$$
where $a_1,\dots,a_n \in \TT = \RR \cup \{-\infty\}$. Observe then for any $\sigma \in S_n$,
$$A_{1,\sigma(1)} \odot A_{2,\sigma(2)} \odot \cdots \odot A_{n, \sigma(n)} = -\infty \iff (\exists i \in [n]): \ A_{i,\sigma(i)} = -\infty.$$
Therefore we may conclude that
\begin{align*}
\tropdet(A) &= \bigoplus_{\sigma \in S_n} (A_{1,\sigma(1)} \odot A_{2,\sigma(2)} \odot \cdots \odot A_{n, \sigma(n)}) \\
&= \max_{\sigma \in S_n} \{A_{1,\sigma(1)} + A_{2,\sigma(2)} + \cdots + A_{n, \sigma(n)}\} \\
&= \max \{ -\infty, \dots, -\infty, \sum_{i=1}^n a_i, -\infty, \dots, -\infty\} \\
&= a_1 + a_2 + \cdots + a_n.
\end{align*}
And so if we write a tropical matrix $A \in G(n \times n)$ as $A_\sigma \odot D(\beta_1,\dots,\beta_n)$ for $\beta_i \in \TT$, we have that
$$\tropdet(A) = \sum_{i=1}^n \beta_i$$
which will be finite (not $-\infty$) if and only if each $\beta_i$ is finite, i.e. $A \in G(n) \subseteq G(n \times n)$, thus $A$ is an invertible tropical matrix.
\end{proof}

The above statement is reminiscent of \cite[Theorem 3.1]{troplinear}; the reader interested in pursuing the tropical linear algebra framework further is directed there.

\begin{remark} \label{badexample}
Observe that this does not hold for a general tropical matrix in $\Mat(n \times n, \TT)$:
$$\tropdet \begin{pmatrix} 0 & 0 \\ 0 & -\infty \end{pmatrix} = (0 \odot -\infty) \oplus (0 \odot 0) = -\infty \oplus 0 = 0,$$
but the above matrix has more than one finite entry in its first row, so it is not in $\G(2)$.
\end{remark}

\begin{definition}[Tropical rank]
We now define for $A \in \G(n \times m)$ its \textit{tropical rank}, which is crafted in analogy with the classical case:

$$\troprank(A) = k \iff \begin{cases} \text{all $(k+1) \times (k+1)$ minors vanish, and} \\ \text{there exists a non-vanishing $k \times k$ minor.} \end{cases}$$
Where we say a minor \textit{vanishes} if it equals $-\infty$.
\end{definition}

The ``correct" notion of rank for tropical matrices is a matter of debate, but the definition above satisfies the properties we require. Other variants, such as \textit{Barvinok rank} and \textit{Kapranov rank}, are related as follows:

\begin{theorem}[\cite{develin2003rank}]
For every matrix $M$ with entries in the tropical semiring,
$$\troprank(M) \leq \text{Kapranov rank}(M) \leq \text{Barvinok rank}(M).$$
Both of these inequalities can be strict.
\end{theorem}

Geometrically, every $(d \times n)$-matrix has an image that is a tropical polytope in $\RR^d/\RR\mathbf{1}$ (see \cite[\S5.2]{tropbook}). The tropical rank of a matrix is the dimension of this tropical polytope \textit{plus one}.

\begin{example}
Unlike working over $\RR$, (tropical) rank is not generally additive over $\TT$:
$$\troprank \begin{pmatrix} 2 & -\infty & -\infty \\ -\infty & -\infty & 3 \end{pmatrix} = 2 \qquad \text{and} \qquad \troprank \begin{pmatrix} 0 & -\infty &-\infty \\ -\infty & 0 & -\infty \\
-\infty & -\infty & -\infty \end{pmatrix} = 2$$
since $\left| \begin{smallmatrix} 2 & -\infty \\ -\infty & 3 \end{smallmatrix} \right| = 5 \neq -\infty$ and $\left| \begin{smallmatrix} 0 & -\infty \\ -\infty & 0 \end{smallmatrix} \right| = 0 \neq -\infty$, yet
$$\begin{pmatrix} 2 & -\infty & -\infty \\ -\infty & -\infty & 3 \end{pmatrix} \odot \begin{pmatrix} 0 & -\infty &-\infty \\ -\infty & 0 & -\infty \\ -\infty & -\infty & -\infty \end{pmatrix} = \begin{pmatrix} 2 & -\infty & -\infty \\ -\infty & -\infty & -\infty \end{pmatrix}$$
but the resulting matrix has tropical rank $1$. Thus for matrices $A \in \G(n \times m)$ and $B \in G(m \times s)$, generally
$$\troprank(A \odot B) \neq \troprank(A) + \troprank(B).$$
However, repositioning the finite entries via a tropical change of basis restores equality. We make this precise below.
\end{example}

\begin{lemma} \label{troprank-props}
Let $A \in \G(n \times m)$ and take $N \in \G(n)$ and $M \in \G(m)$. Then
$$\troprank(A) = \troprank(N \odot A \odot M)$$
and $\troprank(A)$ is the number of columns having at least one finite entry.
\end{lemma}

\begin{proof}
    First we prove that $\troprank(A)$ is the number of columns having at least one finite entry. The matrices in $G(n \times m)$ are only subject to having at most one finite entry in each row. As such, there may be empty (all $-\infty$) columns or rows, which do not contribute the rank. Assume then without loss of generality that $A$ has no such empty rows or columns. Then the dimension of $A$ is $k \times \ell$ with $k \geq \ell$, and since a minor is a determinant of a square matrix, we must have that $\troprank(A)=\ell$. But as $\ell$ was the number of columns having a finite entry, we are done.

    \medskip

    To prove the second statement, by Remark \ref{help}, we may write
    $$N = E_\sigma \odot D(\beta_1,\dots,\beta_n) \quad \text{ and } \quad M = E_\tau \odot D(\gamma_1,\dots,\gamma_m)$$
    where $\beta_i, \gamma_j \in \RR$. This allows us to express the product as
    $$N \odot A \odot M = E_\sigma \odot D(\beta_1,\dots,\beta_n) \odot A \odot E_\tau \odot D(\gamma_1,\dots,\gamma_m).$$
    We may assume without loss of generality that $\beta_i,\gamma_j$ are all zero, as the only aspect that would distinguish rank is whether entries are infinite. But as $D(0,\dots,0)$ is the identity matrix, this simplifies this expression immediately:
    $$E_\sigma \odot E_n \odot A \odot E_\tau \odot E_m = E_\sigma \odot A \odot E_\tau.$$
    This situation is now analogous to multiplying with permutation matrices, where $E_\sigma$ is permuting the rows of $A$ and $E_\tau$ the columns. It follows that $\oplus_{\alpha \in S_k} \odot_{i=1}^k a_{i,\alpha(i)} \neq -\infty$ if and only if $\oplus_{\alpha \in S_k} \odot_{i=1}^k a_{\sigma(i),\tau(\alpha(i))} \neq -\infty$, hence $\troprank(A) = \troprank(N \odot A \odot M)$.
\end{proof}

Many common matrix operations carry over to the tropical setting, but multiplication and addition must be replaced by their tropical counterparts.

\begin{definition}[Direct sum and tensor products of tropical matrices]
Consider two tropical matrices $A = (a_{ij}) \in \Mat(m \times n, \TT)$ and $B = (b_{st}) \in \Mat(k \times \ell, \TT)$, we then define
\begin{itemize}
	\item the \textit{direct sum} to be the $(m+k)\times (n+\ell)$ tropical matrix
$$A \oplus B = \begin{pmatrix} A & -\infty \\ -\infty & B  \end{pmatrix},$$
	\item the \textit{tensor product} to be the $mk \times n\ell$ tropical matrix
$$A \otimes B = \begin{pmatrix} a_{11} \odot B & \cdots & a_{1n} \odot B \\ \vdots & \ddots & \vdots \\ a_{m1} \odot B & \cdots & a_{mn} \odot B  \end{pmatrix}.$$
\end{itemize}
\end{definition}

It follows that the tensor product of tropical matrices expands over their direct sum, i.e.
$$\left( \bigoplus_{i=1}^n A_i \right) \otimes  B = \bigoplus_{i=1}^n (A_i \otimes B),$$
which can be seen by expanding out the entries.

\subsection{Tropical Vector Bundles}
We now introduce the primary objects of study. Most of these notions were introduced or formulated for our purposes in \cite{chern}, to which the reader is directed for further details on tropical vector bundles.

\begin{definition}[Tropical vector bundles] \label{vbdef}
Let $X$ be a rational polyhedral space. A \textit{tropical vector bundle over $X$ of rank $r$} is a rational polyhedral space $F$ together with a morphism $\pi: F \to X$ and a finite open covering $\{U_1, \dots, U_s\}$ of $X$ as well as a homeomorphism $\Phi_i: \pi^{-1}(U_i) \rightarrow U_i \times \RR^r$ for every $1 \leq i \leq s$ such that
\begin{itemize}
	\item[(a)] for all $i$ we have that the diagram
\[\begin{tikzcd}
	{\pi^{-1}(U_i)} & {U_i \times \mathbb{R}^r} \\
	& {U_i}
	\arrow["{\Phi_i}", from=1-1, to=1-2]
	\arrow["{p_1}", two heads, from=1-2, to=2-2]
	\arrow["\pi"', from=1-1, to=2-2]
\end{tikzcd}\]
		commutes, where $p_1: U_i \times \RR^r \to U_i$ is the projection to the first factor,
	\item[(b)] for all $i,j$ the composition $p_j^{(i)} \circ \Phi_i: \pi^{-1}(U_i) \to \RR$ is a regular invertible function, where $p^{(i)}_j: U_i \times \RR^r \to \RR$ is projection onto the $j$th component of $\RR^r$, i.e. $(x,(a_1,\dots,a_r)) \mapsto a_j$,
	\item[(c)] for every $i,j \in \{1,\dots,s\}$ there exists a \textit{transition map} $M_{ij}: U_i \cap U_j \to G(r)$ such that
			$$\Phi_j \circ \Phi_i^{-1}: (U_i \cap U_j) \times \RR^r \rightarrow (U_i \cap U_j) \times \RR^r$$
		is given by $(x,a) \mapsto (x,f_{M_{ij}(x)}(a))$ and the entries of $M_{ij}$ are regular invertible functions on $U_i \cap U_j$ or constantly $-\infty$, 
	\item[(d)] there exist representatives $F_0$ of $F$ and $X_0$ of $X$ such that
			$$F_0 = \{\pi^{-1}(\tau) \ | \ \tau \in X_0\}$$
and $\omega_{F_0}(\pi^{-1}(\tau)) = \omega_{X_0}(\tau)$ for all maximal polyhedral $\tau \in X_0$.
\end{itemize}
An open set $U_i$ with the map $\Phi_i: \pi^{-1}(U_i) \rightarrow U_i \times \RR^r$ is called a \textit{local trivialization} of $F$.
\end{definition}

\begin{remark}
In the original source \cite{chern}, the fibers are all taken to be $\RR = \TT^\times$, but a reader familiar with other work in this area (e.g. \cite{shaw} and \cite{gross2023semi}) will notice that at times the fibers are instead taken to be $\TT$. This hinges on condition $(c)$ above, where we funnel transition maps $M_{ij}$ through $f_{M_{ij}}$. Throughout the work we will at times take fibers that are $\TT$ instead, and this is meant to imply that our transition maps no longer go through this extra step to be re-interpreted as a map on $\RR^n$.
\end{remark}

There has been considerable work with the above definition of tropical vector bundles (introduced by Allermann \cite{chern}). In \cite{gross2022principal}, \cite{gross2023semi}, the authors studied the case where $X$ is an \textit{integral affine manifold} $(X,\text{Aff}_X)$, with $\text{Aff}_X$ a sheaf of affine functions.

Adopting this language, we may then interpret tropical vector bundles of rank $n$ as being $S_n \ltimes \text{Aff}_X^n$-torsors. This terminology continues in the expected way (e.g. morphisms of tropical vector bundles become morphisms of torsors, etc.). Most interestingly though is the following result they obtain, which applies just as well in our setting:

\begin{proposition}[\cite{gross2023semi}]
    The category of tropical vector bundles of rank $n \in \ZZ_{\geq 1}$ on an integral affine manifold $X$ is equivalent to the category of pairs $(\widetilde{X} \overset{\pi}{\longrightarrow} X, \mathcal{L})$ consisting of a degree $n$ free cover $\pi: \widetilde{X} \to X$ and a tropical line bundle $\mathcal{L} \to X$.
\end{proposition}

That is, a tropical vector bundle of rank $n$ on $X$ may be constructed by passing to a covering space $\widetilde{X}$ and assembling line bundle fibers over each point. The simplest example takes $X$ to be a circle with a rank-$2$ bundle and $\widetilde{X}$ a double cover with a line bundle.

In perhaps a different flavor, \cite{tropschemes} generalizes this definition of tropical vector bundles to apply more broadly to tropical/monoidal schemes, proving in Proposition 5.1 that for a pair $(X,\mathcal{O}_X)$ of a topological space with a sheaf (of semirings) of $\TT$-valued functions, that there is an equivalence of categories between topological $\TT$-vector bundles on $X$ (as defined in Definition \ref{vbdef}) and locally free $\mathcal{O}_X$-modules, given by sending a vector bundle to its sheaf of continuous sections. They later prove an analogue of the above equivalence of isomorphism classes of rank-$n$ topological $\TT$-vector bundles with line bundles on $n$-fold covering spaces for locally connected paracompact Hausdorff spaces in \cite[Proposition 5.5]{tropschemes}.

\begin{definition}[Pull-back of vector bundles]
Let $\pi: F \to X$ be a vector bundle of rank $n$ with an open covering $U_1, \dots, U_s$ and transition maps $M_{ij}$. Let $f:Y \to X$ be a morphism of rational polyhedral spaces. Then the \textit{pullback bundle} $\pi': f^*F \to Y$ is the vector bundle we obtain by gluing the patches $f^{-1}(U_1) \times \RR^n, \dots, f^{-1}(U_s) \times \RR^n$ along the transition maps $M_{ij} \circ f$. Hence we have the following commutative diagram
\[\begin{tikzcd}[sep=scriptsize]
	{f^*F} && F \\
	\\
	Y && X
	\arrow["f", from=3-1, to=3-3]
	\arrow["{\pi'}"', dashed, from=1-1, to=3-1]
	\arrow["\pi"', from=1-3, to=3-3]
	\arrow["{f'}", dashed, from=1-1, to=1-3]
\end{tikzcd}\]
where $f'$ and $\pi'$ are locally given by
\[\begin{tikzcd}[row sep=tiny]
	{f': f^{-1}(U_i) \times \mathbb{R}^n} & {U_i \times \mathbb{R}^n} & {\text{and}} & {\pi':f^{-1}(U_i) \times \mathbb{R}^n} & {f^{-1}(U_i)} && {} \\
	{(y,a)} & {(f(y),a)} & {} & {(y,a)} & y. && {}
	\arrow[maps to, from=2-1, to=2-2]
	\arrow[from=1-1, to=1-2]
	\arrow[from=1-4, to=1-5]
	\arrow[maps to, from=2-4, to=2-5]
\end{tikzcd}\]
\end{definition}

\begin{definition}[Subbundles]
Let $\pi: F \to X$ be a vector bundle with an open trivialization $U_1,\dots, U_s$ with corresponding homeomorphisms $\Phi_i$. A subcycle $E \in Z_{\ell}(F)$ is called a \textit{subbundle of rank $r$ of $F$} if $\pi|_E: E \to X$ is a bundle of rank $r$ such that for $1 \leq i \leq s$ we have homeomorphisms
$$\Phi_i|_{(\pi|_E)^{-1}(U_i)}: (\pi|_{E})^{-1}(U_i) \rightarrow U_i \times \langle e_{j_1}, \dots, e_{j_r'} \rangle_\RR$$
for some $1 \leq j_1 < \cdots < j_{r'} \leq r$, where $e_j$ are the standard basis vectors in $\RR^r$.
\end{definition}

Allermann noted the following as Remark 1.13 in \cite{chern}, but we state it as a proposition to emphasize its importance. Together with the definition of quotient bundles below, it implies that all short exact sequences of tropical vector bundles split.

\begin{proposition}
If $\pi: F \to X$ is a vector bundle of rank $n$ with a subbundle $E$ of rank $r$, then there exists a subbundle $E'$ of $F$ of rank $n-r$ such that $F \cong E \oplus E'$.
\end{proposition}

\begin{proof}
Assume such a subbundle $E$ of $F$ exists. Then just consider the natural subbundle $E'$ that locally has the complementary basis in each local trivialization.
\end{proof}

\begin{definition}
Let $\pi: F \to X$ be a vector bundle of rank $n$. We say $F$ is \textit{decomposable} if there exists a subbundle $\pi|_E: E \to X$ of $F$ of rank $1 \leq r < n$. Otherwise, $F$ is said to be \textit{indecomposable}.
\end{definition}

This motivates the following definition.

\begin{definition}[Quotient bundles]
Let $\pi: F \to X$ be a vector bundle and $F \cong S \oplus Q$. Then we define the \textit{quotient bundle $F/S$} to be $Q$.
\end{definition}

These definitions are straightforwardly consistent with one another, i.e. every short exact sequence of tropical vector bundles is split.

\begin{definition}[Direct Sum, Tensor, Dual]
Let $E$ and $F$ be two vector bundles of ranks $e$ and $f$ over $X$ with local trivializations $U_1,\dots,U_s$ and transition maps $\{M_{ij}^E\}$ and $\{M_{ij}^F\}$. Then we have the following operations:
\begin{itemize}
    \item the \textit{direct sum bundle} $E\oplus F \to X$ is the rank $e+f$ bundle having transition maps 
\[\begin{tikzcd}[row sep=0]
	{M_{ij}^{E \oplus F} : \hspace{-.4in}} & { U_i \cap U_j} & {\G(e+f)} \\
	& M_{ij}^{E \oplus F}(p)  & { M_{ij}^E(p) \oplus M_{ij}^F(p)}
	\arrow[from=1-2, to=1-3]
	\arrow[maps to, from=2-2, to=2-3]
\end{tikzcd}\]

    \item the \textit{tensor product bundle} $E \otimes F \to X$ is the rank $ef$ bundle having transition maps
\[\begin{tikzcd}[row sep=0]
	{M_{ij}^{E \otimes F} : \hspace{-.4in}} & { U_i \cap U_j} & {\G(ef)} \\
	& M_{ij}^{E \otimes F}(p)  & { M_{ij}^E(p) \otimes M_{ij}^F(p)}
	\arrow[from=1-2, to=1-3]
	\arrow[maps to, from=2-2, to=2-3]
\end{tikzcd}\]

    \item the \textit{dual bundle} $E^\vee \to X$ is the rank $e$ bundle having transition maps
    \[\begin{tikzcd}[row sep=0]
	{M_{ij}^{E^\vee} : \hspace{-.4in}} & { U_i \cap U_j} & {\G(e)} \\
	& M_{ij}^{\vee}(p)  & { \left(M_{ij}(p)^T\right)^{-1} }
	\arrow[from=1-2, to=1-3]
	\arrow[maps to, from=2-2, to=2-3]
\end{tikzcd}\]
where the inverse is taken under the natural group structure of $\G(e)$.
\end{itemize}
\end{definition}

As a sanity check, we show that if $A \in G(n)$ and $B \in G(m)$, then indeed $A \otimes B \in G(nm)$. This confirms that the above transition maps indeed land in the codomain $G(fe) \subseteq \Mat(fe \times fe, \TT)$. Without loss of generality, we may assume that the finite entries of $A$ run along the diagonals $a_{ii}$ for $1 \leq i \leq n$. We then have that
$$A \otimes B = \begin{pmatrix} a_{11} \odot B & a_{12} \odot B & \cdots & a_{1n} \odot B \\
a_{21} \odot B & a_{22} \odot B & \cdots & a_{2n} \odot B \\
\vdots & \vdots & \ddots & \vdots \\
a_{n1} \odot B & a_{n2} \odot B & \cdots & a_{nn} \odot B  \end{pmatrix} =
\begin{pmatrix} a_{11} \odot B & -\infty & \cdots & -\infty \\
-\infty & a_{22} \odot B & \cdots & -\infty \\
\vdots & \vdots & \ddots & \vdots \\
-\infty & -\infty & \cdots & a_{nn} \odot B  \end{pmatrix}$$
and as $B$ has exactly one finite entry in each row and column, we may then conclude that $A \otimes B \in G(nm)$.

\begin{definition}[Morphisms of vector bundles]
A morphism of vector bundles $\pi_F: F \to X$ of rank $f$ and $\pi_E:E \to X$ of rank $e$ is a morphism $\Psi: F \to E$ of rational polyhedral spaces such that
\begin{itemize}
	\item[(a)] the diagram
\[\begin{tikzcd}[column sep=small]
	F && E \\
	& X
	\arrow["\Psi", from=1-1, to=1-3]
	\arrow["{\pi_F}"', from=1-1, to=2-2]
	\arrow["{\pi_E}", from=1-3, to=2-2]
\end{tikzcd}\]
		commutes, i.e. $\pi_F = \pi_E \circ \Psi$ and
	\item[(b)] there exists local trivializations $U_1,\dots,U_s$ of $F$ and $E$ and associated maps $A_i: U_i \to G(e \times f)$ such that
$$\Phi_i^E \circ \Psi \circ (\Phi_i^F)^{-1}: U_i \times \RR^f \rightarrow U_i \times \RR^e$$
is given by $(x,a) \mapsto (x,f_{A_i(x)}(a))$, and the entries of $A_i$ are regular invertible functions on $U_i$ or constantly $-\infty$.
\end{itemize}
\end{definition}

Consider now a morphism of tropical vector bundles 
\[\begin{tikzcd}[column sep=small]
	\mathcal{E} && \mathcal{F} \\
	& X
	\arrow["\phi", from=1-1, to=1-3]
	\arrow["{\pi_\mathcal{F}}", from=1-3, to=2-2]
	\arrow["{\pi_\mathcal{E}}"', from=1-1, to=2-2]
\end{tikzcd}\]
where $\mathcal{E}$ is rank $e$ and $\mathcal{F}$ is rank $f$. Let $U \subseteq X$ be a local trivialization of both $\mathcal{E}$ and $\mathcal{F}$. We then have a corresponding map $A: U \to \G(f \times e)$, which induces, per Remark \ref{fix}, a map $f_A: U \to \text{Hom}_{\textsf{Set}}(\RR^e, \RR^f)$ such that
$$\Phi^\mathcal{F} \circ \phi \circ (\Phi^\mathcal{E})^{-1}: U \times \RR^e \rightarrow U \times \RR^f$$
where $(p,v) \mapsto (p,f_A(p)(v))$. We now wish to define a notion of how $\phi$ acts `pointwise'.

\begin{definition}
With notation as above, we define the \textit{induced stalk map at $p$} to be the tropical linear map $\phi_p: \TT^e \to \TT^f$ given by $\phi_p(v) = A(p)(v)$, i.e. having $A(p) \in \G(f \times e)$.
\end{definition}

Note that this definition of the induced stalk map is invariant of the choice of local trivialization, up to a change of basis. We prove this below.

\begin{lemma} \label{stalk-invariant}
Let $U,V \subseteq X$ be local trivializations of a tropical vector bundle morphism $\phi: \mathcal{E} \to \mathcal{F}$. We denote by
$$A_U: U \to \G(f \times e) \quad \text{ and } \quad A_V: V \to \G(f \times e)$$
the associated transition maps. Then for $p \in U \cap V$,
$$A_V(p) = M_{U,V}^\mathcal{F}(p) \odot A_U(p) \odot (M_{U,V}^\mathcal{E}(p))^{-1}$$
where we have the transition maps
$$M_{U,V}^\mathcal{F}: U\cap V \to \G(f) \quad \text{ and } \quad M_{U,V}^\mathcal{E}: U\cap V \to \G(e).$$
\end{lemma}

\begin{proof}
This statement is equivalent to the content that the following diagram commutes:
\[\begin{tikzcd}
	{(U \cap V) \times \mathbb{T}^e} && {(U \cap V) \times \mathbb{T}^e} \\
	\\
	{(U \cap V) \times \mathbb{T}^f} && {(U \cap V) \times \mathbb{T}^f}
	\arrow["{\text{id} \times A_U{\Large|}_{U\cap V}}"', from=1-1, to=3-1]
	\arrow["{\text{id} \times A_V{\Large|}_{U\cap V}}", from=1-3, to=3-3]
	\arrow["{\id \times M_{U \cap V}^\mathcal{E}}", from=1-1, to=1-3]
	\arrow["{\id \times M_{U \cap V}^\mathcal{F}}", from=3-1, to=3-3]
\end{tikzcd}\]
which follows from the commutativity property of tropical vector bundle morphisms and that the above transitions are isomorphisms.
\end{proof}

The above allows us to establish what is one of the most important tools for our subsequent work.

When the base space $X$ has sedentary strata (boundary), the entries of $A_i(p)$ that are regular invertible functions on the interior of a face can approach $-\infty$ as $p$ approaches a sedentary face of $X$. At such boundary points, the induced stalk map $\phi_p$ has strictly lower rank than at interior points. This is the mechanism by which degeneracy loci acquire positive codimension.

\begin{corollary} \label{rank-upper-semi}
Given a tropical vector bundle morphism $\phi: \mathcal{E} \to \mathcal{F}$ over a rational polyhedral space $X$, we have an induced rank function
\[\begin{tikzcd}[row sep=tiny]
	{\text{rank($\phi$)}: \hspace{-.4in}} & {|X|} & {\mathbb{Z}} \\
	& p & {\troprank(\phi_p)}
	\arrow[from=1-2, to=1-3]
	\arrow[maps to, from=2-2, to=2-3]
\end{tikzcd}\]
which is locally constant on the interior of each top-dimensional face, but may drop at sedentary strata. In particular, $\rank(\phi)$ is upper semicontinuous on $|X|$.
\end{corollary}

\begin{proof}
Well-definedness follows from Lemma \ref{stalk-invariant} and Lemma \ref{troprank-props}. On the interior of a top-dimensional face $\sigma$ of $X$, the entries of the local matrix $A_i(p)$ are continuous functions valued in $\RR \cup \{-\infty\}$ whose non-$(-\infty)$ entries are regular invertible (hence piecewise-linear on $\RR$). On connected open subsets of the interior, the set of non-$(-\infty)$ entries is constant, so $\rank(\phi)$ is locally constant there.

At a sedentary face $\tau \subseteq |X|$ (where some coordinates of the ambient $\TT^n$ equal $-\infty$), regular invertible functions on the interior that extend continuously to $\tau$ may take the value $-\infty$ at $\tau$, causing additional entries of $A_i(p)$ to become $-\infty$. Since additional entries becoming $-\infty$ can only decrease the tropical rank, $\rank(\phi)$ is upper semicontinuous.
\end{proof}

\begin{remark}
Upper semicontinuity of the rank function means that the sublevel sets $D_k(\phi) = \{p \in |X| \mid \troprank(\phi_p) \leq k\}$ are closed in $|X|$ for each $k$. On the interior, the locally constant behavior implies $D_k(\phi) \cap X^\circ$ is a union of connected components of $X^\circ$. The genuine rank drop occurs at sedentary strata, giving $D_k(\phi)$ additional faces of positive codimension.
\end{remark}

\begin{lemma} \label{rank-morph}
Consider a morphism of tropical vector bundles $\phi: \mathcal{E} \to \mathcal{F}$ over a rational polyhedral space $X$, with $\rank \ \mathcal{E} = e$ and $\rank \ \mathcal{F} = f$. Then the sublevel sets
$$D_k(\phi) = \{p \in |X| \mid \troprank(\phi_p) \leq k\}$$
for $0 \leq k \leq \min(e,f)$ are closed subsets of $|X|$ that inherit the structure of rational polyhedral subspaces of $X$.
\end{lemma}

\begin{proof}
That $D_k(\phi)$ is closed follows from the upper semicontinuity of $\rank(\phi)$ (Corollary \ref{rank-upper-semi}). To see that $D_k(\phi)$ inherits a rational polyhedral structure, pass to a refinement of $X$ on which $\rank(\phi)$ is constant on the interior of each top-dimensional face $\sigma$ (this is possible since the rank function is locally constant on the interior). Then $D_k(\phi)$ is a union of closures of those top-dimensional faces $\sigma$ on which $\rank(\phi)|_{\sigma^\circ} \leq k$, together with any sedentary faces where the rank drops further. It remains to verify balancing. Let $\tau$ be a codimension-one face (ridge) of $D_k(\phi)$. If $\tau$ lies in the interior of $|X|$, then upper semicontinuity forces $\rank(\phi)|_{\tau} \leq k$, and since rank is constant on each adjacent facet interior, every facet $\sigma > \tau$ in $X$ with $\rank(\phi)|_{\sigma^\circ} \leq k$ is in $D_k(\phi)$. But since $\rank(\phi)|_\tau \leq k$ and rank is locally constant on each $\sigma^\circ$, all facets adjacent to $\tau$ satisfy $\rank(\phi)|_{\sigma^\circ} \leq k$ (otherwise $\tau$ would not be a codimension-one face of $D_k(\phi)$, but rather a boundary face). Hence $\tau$ has the same adjacent facets and weights in $D_k(\phi)$ as in $X$, so the balancing sum is inherited. For sedentary ridges, the balancing condition is likewise inherited within the appropriate stratum.
\end{proof}

\section{Tropical Chern Classes}

\begin{definition}[Rational sections of vector bundles]
Let $\pi: F \to X$ be a vector bundle of rank $n$. A \textit{rational section} $s: X \to F$ of $F$ is a continuous map $s: |X| \to |F|$ such that
\begin{enumerate}
	\item[(a)] $\pi \circ s = \id_{|X|}$, i.e. $s$ is a section of $\pi$, and
	\item[(b)] there exists a cover of $X$ by local trivializations and with associated homeomorphisms $\mathcal{U}$ such that for each $(U, \Phi) \in \mathcal{U}$ the components of $\Phi \circ s$ are rational functions on $U$, i.e.
$$\Phi \circ s = (\id_U , s_1 , \dots , s_n): U \rightarrow U \times \RR^n$$
and $s_j \in \KK(U)$ for each $1 \leq j \leq n$.

\end{enumerate}
We say a rational section $s:X \to F$ is \textit{bounded} if the components are bounded over each local trivialization. 
\end{definition}

\begin{remark}
For the remainder of the paper, unless explicitly stated otherwise, all the tropical vector bundles we discuss will be assumed to admit bounded rational sections. 
\end{remark}

An important fact we use routinely is the following: if $\mathcal{L} \to X$ is a line bundle, then a rational section $s: X \to \mathcal{L}$ naturally corresponds to a Cartier divisor $\{(U_i, \widehat{s}_i)\}$, which we denote by $\mathcal{D}(s)$. We state the following result, directing the reader to the source for the proof.

\begin{lemma}[\cite{chern}]
Let $\pi: L \to X$ be a line bundle and let $s_1, s_2: X \to L$ be two bounded rational sections. Then 
$$[\mathcal{D}(s_1)] = [\mathcal{D}(s_2)] \in \Div(X).$$
\end{lemma}

\begin{remark}
Most importantly, the above lemma lets us now associate to any line bundle $\mathcal{L}$ that admits a bounded rational section $s: X \to \mathcal{L}$ a well-defined Cartier divisor class 
$$\mathcal{D}(\mathcal{L}) := [\mathcal{D}(s)] \in \Div(X).$$
\end{remark}
\begin{theorem}[\cite{chern}]
Let $\pi: F \to X$ be a vector bundle of rank $r$ on a simply connected rational polyhedral space $X$. Then $F$ is a direct sum of line bundles, i.e. there exist line bundles $\mathcal{L}_1,\dots,\mathcal{L}_r$ on $X$ such that
$$F \cong \bigoplus_{i=1}^r \mathcal{L}_i.$$
\end{theorem}

\begin{definition}[Global intersection cycles]
Let $\pi: F \to X$ be a vector bundle, $s: X \to F$ be a rational section, and $Y \in Z_\ell(X)$ a subcycle. Then for a positive integer $k \leq \rank \ F$, we define the \textit{global intersection cycle} $s^{(k)} \cdot Y \in Z_{\ell - k}(X)$ on a local trivialization $U$ as
$$(s^{(k)} \cdot Y) \cap U := \sum_{I \in \binom{[r]}{k}} s_{i_1} \cdot s_{i_2} \cdots s_{i_k} \cdot (Y \cap U),$$
where we are summing over all $k$-element subsets $I = \{i_1,\dots,i_k\} \subseteq [r]$.
\end{definition}

\begin{remark}
Observe that if we let $U_1,\dots,U_s$ denote an open covering of $X$ by local trivializations of the bundle $\pi:F \to X$, and we  let $s_{ij}: U_i \to \RR$ denote the $j$th component of $s$ over $U_i$, then
$$(s^{(k)} \cdot Y) \cap U_i := \sum_{1 \leq j_1 < \cdots < j_k \leq n} s_{ij_1} \cdot s_{ij_2} \cdots s_{ij_k} \cdot (Y \cap U_i).$$
Since $s_{i'j} = s_{i\sigma(j) + \phi_j}$ on $U_i \cap U_{i'}$ for some $\sigma \in S_n$ and some regular invertible map $\mathcal{O}^*(U_i \cap U_{i'})$, the intersection products $(s^{(k)} \cdot Y) \cap U_i$ and $(s^{(k)} \cdot Y) \cap U_{i'}$ coincide on $U_i \cap U_{i'}$ and we can glue them together, so this object is well-defined.
\end{remark}

Define the Chow ring of a rational polyhedral space to be $A(X) := \bigoplus_{i\geq 0} A_i(X)$; see \cite{polycomplexes} for a detailed exposition.

\begin{lemma}[\cite{chern}] \label{well-def}
With the same notation as above, let $Y \in Z_\ell(X)$ be a cycle and $\phi \in \KK^*(Y)$ a bounded rational function on $Y$. Then the following holds:
$$s^{(k)} \cdot (\phi \cdot Y) = \phi \cdot (s^{(k)} \cdot Y).$$
\end{lemma}

\begin{theorem}[\cite{chern}] \label{indep}
Let $\pi: F \to X$ be a vector bundle of rank $r$ and $s_1,s_2: X \to F$ two bounded rational sections. Then $s_1^{(k)} \cdot Y$ and $s_2^{(k)} \cdot Y$ are rationally equivalent, i.e.
$$[s_1^{(k)} \cdot Y] = [s_2^{(k)} \cdot Y] \in A(X)$$
holds for all subcycles $Y \in Z_\ell(X)$.
\end{theorem}

\begin{definition}[Chern classes] 
Let $\pi: F \to X$ be a vector bundle admitting \textit{bounded} rational sections. We define the $k$th \textit{Chern class} of $F$ to be the endomorphism
\[\begin{tikzcd}[column sep=scriptsize,row sep=tiny]
	{c_k(F): \hspace{-0.35in}} & {A(X)} && {A(X)} \\
	& {[Y]} && {[s^{(k)} \cdot Y]}
	\arrow[from=1-2, to=1-4]
	\arrow[maps to, from=2-2, to=2-4]
\end{tikzcd}\]
where $s: X \to F$ is any bounded rational section. It follows that $c_0(F) = 1_{A(X)}$ and $c_k(F) = 0$ for $k \not\in [\rank \ F]$. We often will write $c_k(F)\cdot Y$ to denote $c_k(F)(Y)$ to emphasize the reliance on intersection product. The \textit{full Chern class} of a vector bundle $F$ is defined to be $$c(F) := \sum_{p \geq 0} c_p(F) \in A(X).$$
\end{definition}

The class $c_k(F)$ above is well-defined by Lemma \ref{well-def} and independent of the choice of rational section by Theorem \ref{indep}. Furthermore, isomorphic vector bundles have the same Chern classes. We state a tropical analogue of Whitney's formula.

\begin{theorem}[Tropical Whitney's Formula] \label{whitney}
If we have a short exact sequence of finite rank tropical vector bundles over $X$,
$$0 \rightarrow \mathcal{E} \rightarrow \mathcal{F} \rightarrow \mathcal{G} \rightarrow 0,$$
then $c(\mathcal{F}) = c(\mathcal{E})c(\mathcal{G}).$
\end{theorem}

\begin{proof}
We have that every short exact sequence of finite rank tropical vector bundles splits, so $\mathcal{F} \cong \mathcal{E} \oplus \mathcal{G}$. And so we have by  \cite{chern} Theorem 2.11(e),
\begin{align*}
c(\mathcal{E})c(\mathcal{G}) &= \left( \sum_{p\geq 0} c_p(\mathcal{E}) \right) \left( \sum_{p \geq 0} c_p(\mathcal{G}) \right) \\
&= \sum_{p \geq 0} \left( \sum_{i+j = p} c_i(\mathcal{E}) c_j(\mathcal{G}) \right) = \sum_{p \geq 0} c_p(\mathcal{F})
\end{align*}
as we had desired.
\end{proof}

An important consequence of Whitney's formula is that it expresses higher Chern classes of a splitting vector bundle in terms of first Chern classes.

\begin{corollary} \label{whitney-cor}
If a finite rank tropical vector bundle splits as a direct sum of tropical line bundles, i.e. $\mathcal{E} = \bigoplus_{j=1}^e \mathcal{L}_j$, then
$$c(\mathcal{E}) = \prod_{j=1}^e (1+c_1(\mathcal{L}_j))$$
\end{corollary}

\begin{proof}
By definition, $c_0(\mathcal{L}_j) = 1$ for all $j$, and so we apply Whitney's formula.
\end{proof}

This decomposition of higher Chern classes into first Chern classes is a highly desirable property. Even when a vector bundle does not split, one can construct a space where its pullback does, perform computations with the resulting ``virtual" Chern roots, and then push the result back down. This is the core principle behind the splitting construction, which we prove below in the tropical setting.

\begin{theorem}[Properties of Chern classes] \label{chernprops} \label{c1-simple}
Let $\FF \to X$ and $\EE \to X$ be tropical vector bundles admitting bounded rational sections. Furthermore, let $f: Y \to X$ be a morphism of rational polyhedral spaces. Then the following statements hold true:
\begin{enumerate}
	\item classes commute under intersection product, i.e. 
$$c_i(\FF) \cdot c_j(\EE) = c_j(\EE) \cdot c_i(\FF),$$
	\item for all $Z \in A(Y)$ we have that
$$f_*(c_i(f^*\EE)\cdot Z) = c_i(\EE) \cdot f_*(Z),  $$
	\item functoriality, i.e. if $X$ and $Y$ are smooth varieties, then for all $Z \in A(X)$
$$c_i(f^*\FF) \cdot f^*(Y) = f^*(c_i(\FF) \cdot Y),$$
	\item tropical Whitney's formula, i.e. 
$$c_k(\FF \oplus \EE) = \sum_{i+j=k} c_i(\FF) \cdot c_j(\EE),$$
	\item if $\LL$ is a line bundle over $X$, then
$$c_1(\LL) = [\mathcal{D}(\LL)],$$
where $\mathcal{D}(\LL)$ is the Cartier divisor class associated to the line bundle $\LL$,
\item and lastly, for line bundles $\LL$ and $\LL'$ over a rational polyhedral space $X$, we have the following equalities
$$c_1(\mathcal{L}^\vee) = -c_1(\mathcal{L}) \quad \text{ and } \quad c_1(\mathcal{L} \otimes \mathcal{L}') = c_1(\mathcal{L}) + c_1(\mathcal{L}').$$
\end{enumerate}
\end{theorem}

\begin{proof}
See \cite{chern} for (1), (2), (3), and (5). Statement (4) is Theorem \ref{whitney} above. We now prove the two equalities of (6) hold true. We first prove the rational equivalence of sections of these bundles, in other words,
$$\mathcal{D}(\mathcal{L}^\vee) = -\mathcal{D}(\mathcal{L}) \quad \text{ and } \quad \mathcal{D}(\mathcal{L} \otimes \mathcal{L'}) = \mathcal{D}(\mathcal{L}) + \mathcal{D}(\mathcal{L}')$$
in $\Div(X) = \mathcal{K}^*/\mathcal{O}^*$.

\medskip

Let $\sigma:X \to \mathcal{L}$ be an arbitrary bounded section. We now define a section $\tau: X \to \mathcal{L}^\vee$. Let $\{U_1,\dots,U_s\}$ be a cover of $X$ by shared local trivializations of both $\mathcal{L}$ and $\mathcal{L}^\vee$. Define $\tau$ locally on the trivialization $U_i \subseteq |X|$ as 
$$\widehat{\tau}_i := -\widehat{\sigma}_i$$
with $\widehat{\tau}_i$ denoting $p^{(i)} \circ \Phi_i \circ \tau$ (respectively, $\widehat{\sigma}_i := p^{(i)} \circ \Psi_i \circ \sigma$), where
$$\begin{tikzcd}
	{\pi^{-1}(U_i)} && {U_i \times \mathbb{R}} \\
	{U_i} && {\mathbb{R}}
	\arrow["{p^{(i)}}", from=1-3, to=2-3]
	\arrow["{\Phi_i}", from=1-1, to=1-3]
	\arrow["\sigma", from=2-1, to=1-1]
	\arrow["{\widehat{\sigma}_i}", dashed, from=2-1, to=2-3]
\end{tikzcd}$$
We now need only show that $\tau$ is well-behaved on intersections. Observe that on $U_i \cap U_j$ we have that
\begin{align*}
\frac{\widehat{\tau}_i|_{ij}}{\widehat{\tau}_j|_{ij}} &={\widehat{\tau}_i|_{ij}} \odot ({\widehat{\tau}_j|_{ij}})^{-1} \\
&= {\widehat{\tau}_i|_{ij}} - {\widehat{\tau}_j|_{ij}} \\
&=  (-\widehat{\sigma}_i|_{ij}) - (-\widehat{\sigma}_j|_{ij}) \\
&= -(\widehat{\sigma}_i|_{ij} - \widehat{\sigma}_j|_{ij}) = \frac{\widehat{\sigma}_j|_{ij}}{\widehat{\sigma}_i|_{ij}}
\end{align*}
and since $\sigma$ is a bounded rational section of $\mathcal{L}$, we have that $\frac{\widehat{\sigma}_i|_{ij}}{\widehat{\sigma}_j|_{ij}} = \alpha \in \mathcal{O}^*(U_i \cap U_j)$, and so we may conclude that
$$\frac{\widehat{\tau}_i|_{ij}}{\widehat{\tau}_j|_{ij}} = -\alpha \in \mathcal{O}^*(U_i \cap U_j),$$
as desired. So indeed $\tau$ is a bounded rational section of $\mathcal{L}^\vee$, hence $\mathcal{D}(\tau)$ denotes the corresponding Cartier divisor, and we finally have
\begin{align*}
\mathcal{D}(\mathcal{L}) + \mathcal{D}(\mathcal{L}^\vee) &= [\mathcal{D}(\sigma)] + [\mathcal{D}(\tau)] \\
&= [\{ (U_i, \widehat{\sigma}_i) \}] + [\{ (U_i, \widehat{\tau}_i) \}] \\
&= [\{ (U_i, \widehat{\sigma}_i + \widehat{\tau}_i) \}] \\
&= [\{ (U_i, \widehat{\sigma}_i + (-\widehat{\sigma}_i)) \}] = [\{ (U_i,0) \}].
\end{align*}
As $[\{ (U_i,0) \}] \in \Div(X)$ is the identity, this successfully proves that $c_1(\mathcal{L}^\vee) = -c_1(\mathcal{L})$.

\medskip

We now proceed to prove the second statement. Let $\sigma: X \to \mathcal{L}$ and $\tau: X \to \mathcal{L}'$ be two bounded rational sections of our bundles. We now construct a section $\sigma \otimes \tau: X \to \mathcal{L} \otimes \mathcal{L}'$ locally on a cover of $X$ by local trivializations $\{U_1,\dots, U_s\}$ (shared by both bundles) and proceed as we had above. We define on the trivialization $U_i \subseteq |X|$
$$\widehat{(\sigma \otimes \tau)}_i := \widehat{\sigma}_i \otimes \widehat{\tau}_i.$$
Note as both $\widehat{\sigma}_i$ and $\widehat{\tau}_i$ are maps into $G(1)$, this is just equal to $\widehat{\sigma}_i \odot \widehat{\tau}_i$, i.e.  $\widehat{\sigma}_i + \widehat{\tau}_i$. Checking transitions, we see that
\begin{align*}
\frac{\widehat{(\sigma \otimes \tau)}_i|_{ij}}{\widehat{(\sigma \otimes \tau)}_j|_{ij}} &= {\widehat{(\sigma \otimes \tau)}_i|_{ij}}- {\widehat{(\sigma \otimes \tau)}_j|_{ij}} \\
&= (\widehat{\sigma}_i|_{ij} + \widehat{\tau}_i|_{ij}) - (\widehat{\sigma}_j|_{ij} + \widehat{\tau}_j|_{ij}) \\
&= \widehat{\sigma}_i|_{ij} + \widehat{\tau}_i|_{ij} - \widehat{\sigma}_j|_{ij} - \widehat{\tau}_j|_{ij} \\
&= (\widehat{\sigma}_i|_{ij} - \widehat{\sigma}_j|_{ij}) + (\widehat{\tau}_i|_{ij} - \widehat{\tau}_j|_{ij}) = \frac{\widehat{\sigma}_i|_{ij}}{\widehat{\sigma}_j|_{ij}} \odot \frac{\widehat{\tau}_i|_{ij}}{\widehat{\tau}_j|_{ij}}.
\end{align*}
And as $\sigma$ and $\tau$ are both bounded rational sections of line bundles, we have that 
$$\frac{\widehat{\sigma}_i|_{ij}}{\widehat{\sigma}_j|_{ij}} = \alpha \in \mathcal{O}^*(U_i \cap U_j) \quad \text{ and } \quad \frac{\widehat{\tau}_i|_{ij}}{\widehat{\tau}_j|_{ij}} = \beta \in \mathcal{O}^*(U_i \cap U_j).$$
Hence we can conclude that
$$\frac{\widehat{(\sigma \otimes \tau)}_i|_{ij}}{\widehat{(\sigma \otimes \tau)}_j|_{ij}} = \alpha \odot \beta \in \mathcal{O}^*(U_i \cap U_j),$$
and furthermore as both $\alpha$ and $\beta$ are bounded, as must $\alpha \odot \beta = \alpha+\beta$ be. Therefore indeed $\sigma \otimes \tau$ is a section of $\mathcal{L} \otimes \mathcal{L}'$ and $\sigma \otimes \tau \in \Div(X)$. Finally, we can observe that
\begin{align*}
\mathcal{D}(\mathcal{L}) + \mathcal{D}(\mathcal{L'}) &= [\mathcal{D}(\sigma)] + [\mathcal{D}(\tau)] \\
&= [\{(U_i, \widehat{\sigma}_i)\}] + [\{(U_i, \widehat{\tau}_i)\}] \\
&= [\{(U_i, \widehat{\sigma}_i + \widehat{\tau}_i)\}] \\
&= [\{(U_i, \widehat{(\sigma \otimes \tau)}_i\}] = \mathcal{D}(\mathcal{L} \otimes \mathcal{L}').
\end{align*}
So we may therefore conclude that $c_1(\mathcal{L} \otimes \mathcal{L}') = c_1(\mathcal{L}) + c_1(\mathcal{L}')$.
\end{proof}

By Theorem 2.11(f) of \cite{chern}, we have that for a tropical line bundle $\mathcal{L}$ over a rational polyhedral space $X$ that
$$c_1(\mathcal{L}) = \mathcal{D}(\mathcal{L}) = [\mathcal{D}(\sigma)]_\text{rat}$$
where $\sigma$ denotes a bounded rational section of $\mathcal{L}$ and $\mathcal{D}(\sigma)$ is the Cartier divisor class associated to $\sigma$. As all bounded rational sections of a line bundle are rationally equivalent, $\mathcal{D}(\sigma)$ uniquely determines a Cartier divisor class associated to $\mathcal{L}$ denoted $\mathcal{D}(\mathcal{L})$.

\section{Tropical Splitting Principle}

\subsection{Tropical projectivization} One construction we need is the tropical analogue of projective space; see Section 3 of \cite{main} for further details.

\begin{definition}[Tropical projective space]
The \textit{tropical projective space} $\TT \PP^n$ is defined to be the quotient  $(\TT^{n+1})^\times/\sim$, where the equivalence relation is given by
$$(x_0,\dots,x_n) \sim (y_0,\dots,y_n) \iff (\exists \lambda \in \TT^{\times}): \ (x_0,\dots,x_n) = (\lambda \odot y_0,\dots, \lambda \odot y_n).$$
Note that $\TT^\times = \TT \setminus \{-\infty \}$ and the above condition asks that $x_i = \lambda + y_i$ for all $0 \leq i \leq n$. We denote elements (equivalence classes) in $\TT\PP^n$ by $[x_0:\cdots:x_n]$.
\end{definition}

\begin{remark}
We may cover $\TT\PP^n$ with $(n+1)$ affine charts $\TT^n$, in complete analogy with the classical setting: with $\odot$ playing the role of multiplication and $0$ the multiplicative identity, if $x_i \neq -\infty$ we have
\begin{align*}
[x_0:\cdots:x_i:\cdots:x_n] &= (-x_i) \odot [x_0:\cdots: x_{i-1} :x_i: x_{i+1} : \cdots:x_n] \\
&= [\underbrace{x_0 - x_i}_{y_1} : \cdots : \underbrace{x_{i-1} - x_i}_{y_{i}} : 0 : \underbrace{x_{i+1} - x_i}_{y_{i+1}} : \cdots : \underbrace{x_n - x_i}_{y_n}]
\end{align*}
and so we have compatible canonical embeddings $\TT^n \hookrightarrow \TT\PP^n$ for each choice of $x_i \neq -\infty$ for $0 \leq i \leq n$, having an inverse given by 
$$[x_0:\cdots:x_n] \mapsto (y_1,\dots,y_n)$$
on $U_i = \{x \in \TT\PP^n \ | \ x_i \neq -\infty \}$, which cover $\TT\PP^n$.
\end{remark}

\begin{figure}[h]
    \centering
    \includegraphics[width=.4\linewidth]{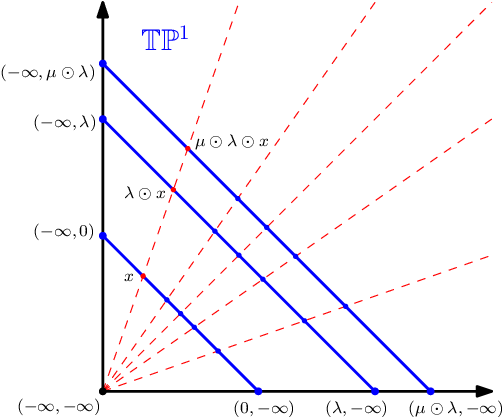}
    
    \caption{The tropical projective line $\TT\PP^1$, with rays emanating from the origin analogous to the classical projective line.}
\end{figure}

This space has a canonical bundle on it that reflects the Serre twisting sheaf $\mathcal{O}_{\PP^n}(1)$ one encounters in the classical setting. 

\begin{definition}
We define the tropical line bundles $\OO_{\tp{n}}(d) \to \tp{n}$, for $d \in \ZZ$, by the local data of
\begin{itemize}
	\item local trivializations over the $(n+1)$-many affine charts $U_i$, $0 \leq i \leq n$, where 
			$$U_i = \{ x \in \tp{n} \ | \ x_i \neq -\infty \},$$
	\item transition functions $M_{ij}: U_i \cap U_j \to G(1) = \RR$ given by\footnote{Importantly note that these are \textit{tropical} operations!}
			$$M_{ij}(x) :=  \left( \frac{x_i}{x_j} \right)^{\odot d} = d(x_i-x_j), $$
	and hence it follows that $\Phi_j \circ \Phi_i^{-1}: (U_i \cap U_j) \times \TT \to (U_i \cap U_j) \times \TT$ is given by
			$$(x,t) \mapsto (x,M_{ij}(x)\odot t) = \left(x, \left( \frac{x_i}{x_j} \right)^{d} \odot t\right) = (x,t+d(x_i-x_j)).$$
\end{itemize}
Note that $\mathcal{O}_\tp{n}(0)$ is just the trivial line bundle $\tp{n}\times \TT$.

\end{definition}

\begin{figure}[h]
    \centering
    \includegraphics[width=.3\linewidth]{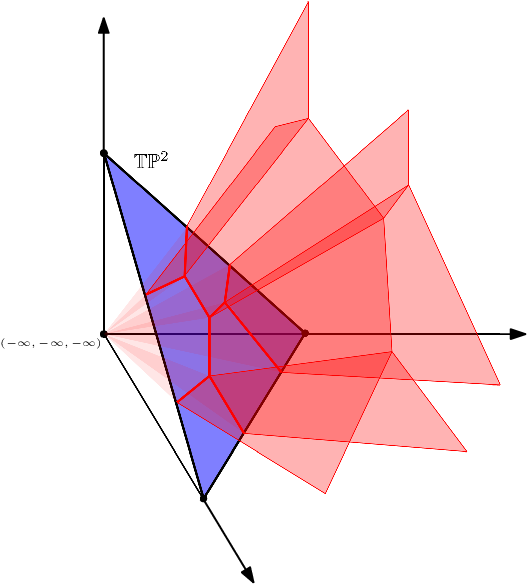}
    \caption{The pullback of $\OO_{\TT\PP^2}(1)$ along an embedding of a curve into $\TP^2$.}
\end{figure}

\begin{theorem} \label{serre-twist}
The tropical line bundle $\OO_\tp{n}(d) \to \tp{n}$ has a first Chern class of
$$c_1(\OO_\tp{n}(d)) = d.$$
\end{theorem}

\begin{proof}
Observe that 
$$\OO_\tp{n}(1)^\vee \cong \OO_\tp{n}(-1) \quad \text{ and } \quad \OO_\tp{n}(d) \cong \OO_\tp{n}(1)^{\otimes d} = \underbrace{\OO_\tp{n}(1) \otimes \cdots \otimes \OO_\tp{n}(1)}_{\text{$d$-many}}$$
for $d \geq 1$. It suffices to show that $c_1(\OO_\tp{n}(1)) = 1$, as the result then follows by properties of Chern classes of line bundles under dualizing and tensoring (see Theorem \ref{chernprops}).

\medskip

We claim that $f = 0 \oplus x_1 \oplus \cdots \oplus x_n = \max(0, x_1, \ldots, x_n)$ defines a bounded rational section of $\OO_\tp{n}(1)$. By Theorem \ref{chernprops}(5), we then have $c_1(\OO_\tp{n}(1)) = [\DD(f)]$, where $\DD(f)$ denotes the associated Cartier divisor.

First observe that $f$ is well-defined on projective equivalence classes, as
$$f(\lambda \odot x) = f([\lambda \odot 0: \lambda \odot x_1 : \cdots : \lambda \odot x_n]) =  \max\{ \lambda, \lambda + x_1, \dots, \lambda + x_n \} = \lambda \odot f(x)$$
for any $\lambda \in \TT^\times = \RR$. On the standard affine chart $U_i$ (where $x_i = 0$), the local representative of $f$ is
$$f_i = \max(x_0, \ldots, x_{i-1}, 0, x_{i+1}, \ldots, x_n),$$
which is a bounded rational function on $U_i$ (as it is a tropical polynomial). We now verify the transition condition: on $U_i \cap U_j$, passing from the $U_i$-normalization (where $x_i = 0$) to the $U_j$-normalization (where $x_j = 0$) subtracts $x_j$ from every coordinate, so
$$f_j = f_i - x_j = f_i + (x_i - x_j) = f_i + M_{ij}(x),$$
where $M_{ij}(x) = x_i - x_j$ is the transition function of $\OO_\tp{n}(1)$. Hence $f$ is indeed a bounded rational section of $\OO_\tp{n}(1)$.

\medskip

It remains to show that $[\DD(f)] = 1$. On each chart $U_i \cong \RR^n$, the Weil divisor $\div(f_i)$ is the locus where $f_i$ fails to be linear, namely the standard tropical hyperplane --- the $(n-1)$-dimensional balanced polyhedral complex where the maximum $\max(x_0, \ldots, 0, \ldots, x_n)$ is attained by at least two terms. This represents the degree-$1$ generator of $A_{n-1}(\tp{n})$, so $c_1(\OO_\tp{n}(1)) = [\DD(f)] = 1$.
\end{proof}

Similarly, various bundles from the tropical toric setting can be reinterpreted in the context of \cite{chern}. For more on tropical toric varieties, see \S 5.3.3 of \cite{main}.

\begin{proposition}[Birkhoff--Grothendieck Theorem] \label{p1-split}
    Every rank $n$ tropical vector bundle $E \to \tp{1}$ splits into a direct sum of line bundles, i.e.
    $$E \cong \bigoplus_{i=1}^n \OO(d_i).$$
\end{proposition}

\begin{proof}
Let $E \to \tp{1}$ be a rank $n$ tropical vector bundle and $\{U_1,\dots, U_s\}$ the associated cover by local trivializations. By Lemma 3.2 of \cite{chern}, we may assume without loss of generality that there exists at most one non-trivial transition map $M_{ij}: U_i \cap U_j \to G(n)$. Up to a (tropical) change of basis, we may assume that this transition matrix takes the form
$$M_{ij} = \begin{pmatrix} g_1 & -\infty & \cdots & -\infty \\
-\infty & g_2 & \cdots & -\infty \\
\vdots & \vdots & \ddots & \vdots \\
-\infty & -\infty & \cdots & g_n  \end{pmatrix}$$
where $g_i \in \mathcal{O}^*(U_i \cap U_j) \cong \mathcal{O}^*(\RR)$ is a regular invertible function. That is to say, $g_i(x) = d_i \odot x$ for integers $d_i$. It then follows that we may decompose $E$ as the direct sum.
\end{proof}

\begin{remark}
    Note this same result has been obtained independently in a different manner (namely by pushing forward line bundles on covers of metric graphs in the work of \cite{gross2022principal}).
\end{remark}

\begin{definition}[projectivization of a bundle]
Let $\pi: \EE \to X$ be a tropical vector bundle of rank $n$. The \textit{projectivization of $\EE$} is a pair $(\TP(\EE), \phi)$ of a rational polyhedral space $\TP(\EE)$ (a ``tropical fiber bundle") and morphism $\phi: \TP(\EE) \to X$ (the projection of the bundle) defined as follows:
\begin{itemize}
    \item $\TP(\EE)$ is the fiber bundle resulting from projectivizing the stalks $\EE_p \cong \TT^n$ of $\EE$, i.e. 
    $$\phi^{-1}(p) := \TP(\EE_p) = (\EE_p \setminus \{ -\infty \})/\sim \ \cong \ \TP^{n-1}$$
    where $v \sim v'$ in $\EE_p\setminus \{-\infty\}$ if there exists some $\lambda \in \TT^\times$ such that $v = \lambda \odot v'$. 
    \item On the level of trivializations $U_i$ of $\EE$, we have
            $$\phi^{-1}(U_i) \cong U_i \times \TP^{n-1},$$
    with transition maps on $U_i \cap U_j$ given by $[M_{ij}] \in \PP G(n)$, where
    $$\PP G(n) := G(n) / (A \sim \lambda \odot A)$$
    for $\lambda \in \TT^\times$. Hence it follows that $\phi$ is locally on a trivialization $U$ given by
    $\phi(p,[x]) = p$
    where $p \in U$ and $[x] \in \TP(\EE_p) \cong \TP^{n-1}$.
\end{itemize}
\end{definition}

\begin{remark}
    The total space $\TP(\EE)$ is a rational polyhedral space whose sedentary strata arise from two sources: the boundary of the base $X$ (sedentary strata of $X$) and the boundary of the fibers $\TP^{n-1}$ (sedentary strata of tropical projective space). That is, $\TP(\EE)$ naturally has the structure of a ``tropical fiber bundle" with boundary. The intersection theory on $\TP(\EE)$ is that of rational polyhedral spaces as developed in \cite{main}.
\end{remark}

Our motivation for constructing a notion of projectivization in this context is due to the following important property: let us first consider the pullback of a tropical vector bundle $\EE$ of rank $n$ along $\phi$, as follows:
\[\begin{tikzcd}
	{\varphi^*(\mathcal{E})} & {\mathcal{E}} \\
	{\mathbb{T}\mathbb{P}(\mathcal{E})} & X
	\arrow["{\hat{\pi}}"', dashed, from=1-1, to=2-1]
	\arrow["{\hat{\varphi}}", dashed, from=1-1, to=1-2]
	\arrow["\lrcorner"{anchor=center, pos=0.125}, draw=none, from=1-1, to=2-2]
	\arrow["\varphi", from=2-1, to=2-2]
	\arrow["\pi", from=1-2, to=2-2]
\end{tikzcd}\]

The pullback $\phi^*(\EE)$ is also of rank $n$, but notably has a canonical line subbundle $\LL_\EE$, i.e. we associate to the point $(p,[L]) \in \TP(\EE)$ the vectors $v$ in the $1$-dimensional subspace $L \subseteq \EE_p \cong \TT^n$. Let us make this precise. Letting $\QQ_\EE$ denote the quotient $\phi^*(\EE)/\LL_\EE$, we have a canonical short exact sequence as follows.

\begin{lemma} \label{ses}
    Let $\EE \to X$ be a rank $n$ tropical vector bundle. Then
\[\begin{tikzcd}
	0 & {\LL_\EE} & {\phi^*(\EE)} & {\QQ_\EE} & 0.
	\arrow[from=1-4, to=1-5]
	\arrow[from=1-3, to=1-4]
	\arrow[from=1-2, to=1-3]
	\arrow[from=1-1, to=1-2]
\end{tikzcd}\]
    is a short exact sequence of tropical vector bundles over $\TP(\EE)$.
\end{lemma}

\begin{proof}
We show that $\LL_\EE$ is a rank $1$ subbundle of $\phi^*(\EE)$ in the tropical sense, from which it follows that $\QQ_\EE$ is the corresponding quotient and the sequence is exact. Let $U_1,\dots,U_s$ be an open cover by trivializations of $\TT\PP(\EE)$. As $\LL_\EE \to \TP(\EE)$ is a line bundle, we need only show that the maps
$$\Phi_i|_{\hat{\pi}|_{\LL_\EE}^{-1}(U_i)}: (\hat{\pi}|_{\LL_\EE}^{-1})(U_i) \rightarrow U_i \times \langle e_j\rangle_{\RR}$$
for $1 \leq i \leq s$ are isomorphisms, where $j \in [n]$ is fixed with $\EE_p \cong \langle e_1,\dots,e_n \rangle_\RR$ for $p \in U_i$. By virtue of $\LL_\EE$ being a line bundle, we have that 
$$(\hat{\pi}|_{\LL_\EE}^{-1})(U_i) \cong U_i \times L^\times \cong U_i \times \RR,$$
where $L^\times$ denotes the units of $\TT \cong L \subseteq \EE_p$, which is just $\RR$. As we also have that there is a natural inclusion $L \subseteq \EE_p$, by taking $\langle e_1 \rangle_\RR$ to be a basis vector of $L$ without loss of generality, this gives us a natural isomorphism
$$U_i \times \langle e_1 \rangle_\RR \rightarrow  (\hat{\pi}|_{\LL_\EE}^{-1})(U_i) \cong U_i \times \RR,$$
where $\hat{\pi}^{-1}(U_i) \cong U_i \times \langle e_1,\dots,e_n \rangle_\RR$, providing precisely what was needed.
\end{proof}

\begin{definition}
    We define the canonical line subbundle $\LL_\EE$ of $\phi^*(\EE)$ above to be $\OO_\EE(-1)$. The dual of this bundle is denoted by
    $$\OO_\EE(1) := \OO_\EE(-1)^\vee$$
    and we will call $\OO_\EE(1)$ the \textit{universal} (or \textit{tautological}) \textit{line sub-bundle} on $\TP(\EE)$.
\end{definition}

\begin{lemma} \label{splitpf}
    Let $\EE \to X$ be a rank $n+1$ tropical vector bundle. Then the induced pullback map
    $$\phi^*: A_k(X) \rightarrow A_{k+n}(\TT\PP(\EE))$$
    is a split monomorphism.
\end{lemma}

\begin{proof}
    We show that $\beta \mapsto \phi_*(c_1(\OO_\EE(1))^{n} \cdot \beta)$ for $\beta \in \im \phi^* \subseteq A_{k+n}(\TP(\EE))$ is an inverse of $\phi^*$. It suffices to prove that $\phi_*(c_1(\OO_\EE(1))^{n} \cdot [\TP(\EE)]) = [X]$.

    \medskip

    The key input is the \textit{projection formula} (\cite{tropmanifolds}; cf.\ Theorem \ref{chernprops}(2) for the special case of pullback Chern classes): for $\alpha \in A(\TP(\EE))$ and $\gamma \in A(X)$,
    $$\phi_*(\alpha \cdot \phi^*(\gamma)) = \phi_*(\alpha) \cdot \gamma.$$
    It therefore suffices to show that $\phi_*(c_1(\OO_\EE(1))^{n} \cdot [\TP(\EE)]) = [X]$, as then for $\beta = \phi^*(\gamma)$ we have
    $$\phi_*(c_1(\OO_\EE(1))^{n} \cdot \phi^*(\gamma)) = \phi_*(c_1(\OO_\EE(1))^{n} \cdot [\TP(\EE)]) \cdot \gamma = [X] \cdot \gamma = \gamma.$$

    \medskip

    To establish this, we use fiber integration. Each fiber of $\phi: \TP(\EE) \to X$ is isomorphic to $\TP^n$ as a rational polyhedral space (with boundary given by the sedentary strata of $\TP^n$). On $\TP^n$, Theorem \ref{serre-twist} gives $c_1(\OO_{\tp{n}}(1)) = 1 \in A_{n-1}(\TP^n)$, and iterating,
    $$c_1(\OO_{\tp{n}}(1))^{n} \cdot [\TP^n] = [\text{pt}] \in A_0(\TP^n),$$
    where the intersection takes place in the Chow ring of $\TP^n$ as a rational polyhedral space. Since the restriction of $\OO_\EE(1)$ to each fiber is $\OO_{\TP^n}(1)$, the pushforward $\phi_*$ performs fiberwise integration, yielding
    $$\phi_*(c_1(\OO_\EE(1))^{n} \cdot [\TP(\EE)]) = [X],$$
    as needed. The intersection theory on $\TP^n$ (including its sedentary strata) is that of rational polyhedral spaces \cite{main}; see also \cite{tropmanifolds} for the projection formula and pushforward in the tropical setting.
\end{proof}

\begin{corollary}[Tropical Splitting Principle] \label{splitting}
For a tropical vector bundle $\pi: \mathcal{E} \to X$ of rank $n$ over a rational polyhedral space $X$, there is an associated rational polyhedral space $Y$ and a morphism $f: Y \to X$ such that
\begin{enumerate}
	\item the pullback bundle $f^*\mathcal{E}$ is a direct sum of tropical line bundles, and
	\item the induced pullback map $f^*: A(X) \to A(Y)$ is injective.
\end{enumerate}
\end{corollary}

\begin{proof}
We proceed with the classical \textit{splitting construction} adapted to our situation. At each step, $\TP(\EE_i)$ is a rational polyhedral space (with boundary from both the base and the fibers), and Lemma \ref{splitpf} applies via the projection formula. We construct $f$ by induction on the rank of $\EE$. Projectivization provides

\[\begin{tikzcd}
	{\varphi^*(\mathcal{E})} & {\mathcal{E}} \\
	{\mathbb{T}\mathbb{P}(\mathcal{E})} & X
	\arrow[""', dashed, from=1-1, to=2-1]
	\arrow["", dashed, from=1-1, to=1-2]
	\arrow["p", from=2-1, to=2-2]
	\arrow["\pi", from=1-2, to=2-2]
\end{tikzcd}\]

\noindent where the pullback $p^*\mathcal{E}$ fits into the short exact sequence
\begin{equation*}
0  \longrightarrow \mathcal{O}_{\TP (\mathcal{E})}(-1) \longrightarrow p^*\mathcal{E} \longrightarrow \mathcal{Q} \longrightarrow 0.
\end{equation*}
If $\mathcal{E}$ has rank $2$, then the quotient bundle $\mathcal{Q} = p^*(\mathcal{E})/\mathcal{O}_{\TP(\mathcal{E})} (-1)$ is a line bundle. Therefore we may take $Y = \TP(\EE)$ as $p^*\EE \cong \mathcal{O}_{\TP(\EE)} (-1) \oplus \mathcal{Q}$ and $p^*$ is injective by Lemma \ref{splitpf}. Otherwise, we may iterate this construction by repeated projectivization and quotienting:

\[\begin{tikzcd}[column sep=small]
	{\overbrace{\varphi_n^*(\mathcal{E}_{n-1})/\mathcal{O}_{{\mathcal{E}}_{n-1}}(-1)}^{\mathcal{E}_n}} && {\overbrace{\varphi_2^*(\mathcal{E}_1)/\mathcal{O}_{{\mathcal{E}}_1}(-1)}^{\mathcal{E}_2}} & {\overbrace{\varphi_1^*(\mathcal{E}_0)/\mathcal{O}_{{\mathcal{E}}_0}(-1)}^{\mathcal{E}_1}} & {\mathcal{E}_0=\mathcal{E}} \\
	{\underbrace{\mathbb{T}\mathbb{P}(\mathcal{E}_n)}_{Y_n}} & \cdots & {\underbrace{\mathbb{T}\mathbb{P}(\mathcal{E}_1)}_{Y_2}} & {\underbrace{\mathbb{T}\mathbb{P}(\mathcal{E}_0)}_{Y_1}} & X
	\arrow[from=1-5, to=2-5]
	\arrow["{p_1}"', from=2-4, to=2-5]
	\arrow[from=1-4, to=2-4]
	\arrow["{p_2}"', from=2-3, to=2-4]
	\arrow["{p_3}"', from=2-2, to=2-3]
	\arrow[from=1-3, to=2-3]
	\arrow["{p_n}"', from=2-1, to=2-2]
	\arrow[from=1-1, to=2-1]
\end{tikzcd}\]
Then taking $Y = Y_n$, we have that the pullback decomposes into a direct sum  $f^*\EE = \bigoplus_{j=1}^n \mathcal{L}_j$, and that $f^*: A(X) \to A(Y)$ is the desired injective map. 

\end{proof}

This allows us to write
$$c_{[t]}(\EE) = \prod_{j=1}^n (1+c_1(\LL_j)t)$$
when computing with Chern classes. We call $c_1(\LL_1), \dots, c_1(\LL_n)$ the \textit{virtual Chern roots} of $\EE$; they live in $A(Y)$ rather than $A(X)$, but the injectivity of $f^*$ ensures that identities proved in $A(Y)$ descend to $A(X)$.

\section{Porteous formula}

\subsection{Sylvester determinants}

The goal of Porteous' formula is to express the class of the $k$th degeneracy locus of a morphism $\phi: \mathcal{E} \to \mathcal{F}$ of tropical vector bundles over $X$, $[{D}_k^\phi(\mathcal{E}, \mathcal{F})] \in A_{(e-k)(f-k)}(X)$, in terms of the Chern classes of $\mathcal{E}$ and $\mathcal{F}$ alone. This expression involves the Sylvester determinant, which we now define.

\begin{definition}[Sylvester determinant]
Let $\mathbf{c} = (c_0,c_1,c_2,\dots)$ be a sequence of elements in a commutative ring $R$. For integers $e,f$, we define the \textit{Sylvester determinant} of $\mathbf{c}$ of order $f$ and degree $e$ to be the element

$$\Delta_f^e(\mathbf{c}) = \det \begin{bmatrix} c_f & c_{f+1} & \cdots & c_{f+e-1} \\
c_{f-1} & c_f & \cdots & c_{f+e-2} \\
\vdots & \vdots & \ddots & \vdots \\
c_{f-e+1} & c_{f-e+2} & \cdots & c_f \end{bmatrix}$$

in the commutative ring $R$. We write $S_f^e(\mathbf{c})$ for the $e \times e$-matrix above.
\end{definition}

Let $a(t) = \sum_{i=0}^e a_it^i$ and $b(t) = \sum_{j=0}^f b_jt^j$ be two polynomials of degrees $e$ and $f$ over a polynomial ring $R[t]$. Assume that $a(0) = a_0 = 1$ and $b(0) = b_0 = 1$. We then have over the appropriate ring extension $R \subseteq S$ that we may factor these polynomials into linear terms
$$a(t) = \prod_{i=1}^e (1+\alpha_i t) \quad \text{ and } \quad b(t) = \prod_{j=1}^f (1+\beta_j t)$$
for coefficients $\alpha_i, \beta_j \in S$.

\begin{lemma} \label{sylvester}
In this setting,
$$ \Delta_f^e\left( \frac{b(t)}{a(t)} \right) = \prod_{\begin{smallmatrix} 1 \leq i \leq e \\ 1 \leq j \leq f \end{smallmatrix}  } (\beta_j - \alpha_i)$$
where $\frac{b(t)}{a(t)}$ is identified with its sequence of coefficients in the ring of power series.
\end{lemma}

\begin{proof}
This is proven as Proposition 12.2 of \cite{3264}. 
\end{proof}

\subsection{Degeneracy Loci}

Now given a morphism $\phi: \mathcal{E} \to \mathcal{F}$ of tropical vector bundles over a rational polyhedral space $X$, we have by Corollary \ref{rank-upper-semi} that the rank function $\rank(\phi): |X| \to \ZZ$ is upper semicontinuous. By Lemma \ref{rank-morph}, the sublevel sets inherit a rational polyhedral structure from $X$. We now define degeneracy loci and exhibit them as subcycles of $X$.

\begin{definition}
Let $X$ be a rational polyhedral space and $\phi: \mathcal{E} \to \mathcal{F}$ a morphism of tropical vector bundles of ranks $e$ and $f$ over $X$. We define the \textit{$k$th tropical degeneracy locus} to be
$$D_k(\phi) = \{ p \in X \ | \  \troprank(\phi_p) \leq k \}.$$
\end{definition}

\begin{proposition} \label{degen-subcycle}
With notation as above, $D_k(\phi)$ is a rational polyhedral subspace of $X$. On the interior $X^\circ$, $D_k(\phi) \cap X^\circ$ is a union of connected components. At sedentary strata of $X$, the rank can drop, giving $D_k(\phi)$ additional faces of positive codimension. The expected codimension of $D_k(\phi)$ is $(e-k)(f-k)$.
\end{proposition}

\begin{proof}
By Lemma \ref{rank-morph}, $D_k(\phi)$ inherits the structure of a rational polyhedral subspace of $X$. On the interior, rank is locally constant by Corollary \ref{rank-upper-semi}, so $D_k(\phi) \cap X^\circ$ is a union of connected components. At sedentary strata, entries of the local matrix $A_i(p)$ can degenerate to $-\infty$, causing the tropical rank to drop below that of nearby interior points. These sedentary faces give $D_k(\phi)$ faces of positive codimension. The expected codimension $(e-k)(f-k)$ is the tropical analogue of the classical expectation: at a generic point of $D_k(\phi)$, exactly $(e-k)(f-k)$ independent conditions are imposed by requiring the rank to be at most $k$.
\end{proof}

Throughout this section, let us assume we have a morphism of tropical vector bundles $\phi: \mathcal{E} \to \mathcal{F}$ over a rational polyhedral space $X$, where $\rank \ \mathcal{E} = e$ and $\rank \ \mathcal{F} = f$.

\subsection{The Hom bundle}

We now establish the key link between morphisms and sections: a morphism $\phi: \EE \to \FF$ is equivalently a bounded rational section of the Hom bundle $\Hom(\EE,\FF) \cong \EE^\vee \otimes \FF$.

\begin{proposition} \label{hom-section}
Given a morphism $\phi: \EE \to \FF$ of tropical vector bundles over $X$ whose matrix entries are bounded (i.e., the finite entries of each local representative $A_i: U_i \to \G(f \times e)$ are bounded functions on $U_i$), there is a canonical bounded rational section $\sigma_\phi$ of $\EE^\vee \otimes \FF$.
\end{proposition}

\begin{remark}
The bounded-entries hypothesis holds automatically when $|X|$ is compact, since continuous functions on compact sets are bounded. In particular, it is satisfied in all applications where the base space is a compact rational polyhedral space (e.g., tropical curves, tropical toric varieties).
\end{remark}

\begin{proof}
On a local trivialization $U_i$, the morphism $\phi$ is represented by $A_i: U_i \to \G(f \times e)$ whose entries are regular invertible functions on $U_i$ or constantly $-\infty$. We define $\sigma_\phi$ locally by viewing $A_i(p)$ as an element of $\TT^{ef}$ (listing the $ef$ entries of the matrix). The transition behavior of $\phi$ (Lemma \ref{stalk-invariant}),
$$A_j(p) = M_{ij}^\FF(p) \odot A_i(p) \odot (M_{ij}^\EE(p))^{-1},$$
is precisely the transition law for sections of $\EE^\vee \otimes \FF$, since the transition maps of $\EE^\vee \otimes \FF$ are $(M_{ij}^\EE)^{-T} \otimes M_{ij}^\FF$, which acts on a matrix $A \in \TT^{ef}$ as $M_{ij}^\FF \odot A \odot (M_{ij}^\EE)^{-1}$. The entries of $A_i$ that are not constantly $-\infty$ are bounded by hypothesis, so $\sigma_\phi$ is a bounded rational section.
\end{proof}

\begin{proposition} \label{zero-locus}
$D_0(\phi) = Z(\sigma_\phi)$, where $Z(\sigma_\phi)$ denotes the zero locus of $\sigma_\phi$ (the locus where all $ef$ components equal $-\infty$).
\end{proposition}

\begin{proof}
By definition, $D_0(\phi) = \{p \in X \mid \troprank(\phi_p) = 0\}$, which is the set of points where $A_i(p)$ has all entries equal to $-\infty$. This is precisely the vanishing locus of $\sigma_\phi$.
\end{proof}

\begin{proposition} \label{top-chern-degen}
If $D_0(\phi)$ has the expected codimension $ef$ in $X$, then
$$[D_0(\phi)] = c_{ef}(\EE^\vee \otimes \FF) \cdot [X].$$
\end{proposition}

\begin{proof}
By definition of Chern classes (cf.~the definition of $c_k$ above), $c_{ef}(\EE^\vee \otimes \FF) \cdot [X] = [\sigma^{(ef)} \cdot X]$ for any bounded rational section $\sigma$ of $\EE^\vee \otimes \FF$. Since $\sigma_\phi$ is such a section (Proposition \ref{hom-section}), by the independence of Chern classes from the choice of section (Theorem \ref{indep}):
$$c_{ef}(\EE^\vee \otimes \FF) \cdot [X] = [\sigma_\phi^{(ef)} \cdot X].$$
We claim that $[\sigma_\phi^{(ef)} \cdot X] = [D_0(\phi)]$. For support: on the complement $|X| \setminus D_0(\phi)$, at least one component $(\sigma_\phi)_j$ is a regular invertible function in each local trivialization (since $p \notin D_0(\phi)$ means not all entries of $A_i(p)$ are $-\infty$). But the Weil divisor of a regular invertible function vanishes, so the iterated intersection $\sigma_\phi^{(ef)} \cdot X$ receives no contribution from $|X| \setminus D_0(\phi)$, and is therefore supported on $Z(\sigma_\phi) = D_0(\phi)$. For the cycle structure: the expected-codimension hypothesis ensures $\dim D_0(\phi) = \dim X - ef$, matching the dimension of $\sigma_\phi^{(ef)} \cdot X$, so the intersection is proper and $[\sigma_\phi^{(ef)} \cdot X] = [D_0(\phi)]$ as cycles in $A_{\dim X - ef}(X)$.
\end{proof}

\subsection{The rank-$0$ case} In the base case of Porteous' formula, we have
$$D_0(\phi) = \{p \in X \ | \  \troprank(\phi_p) = 0\}.$$
The $0$th degeneracy locus consists of all points $p \in X$ where $\phi_p$ is an $(f \times e)$-matrix with all entries equal to $-\infty$ (the tropical analogue of a zero matrix). Equivalently,
$$D_0(\phi) = \{p \in X \ | \ \phi_p: \mathcal{E}_p \to \mathcal{F}_p \text{ is identically $-\infty$}\}.$$
Note, this \textit{is not} the same as asking that $f_{\phi_p}: \RR^e \to \RR^f$ is identically zero! Above, we are specifically referencing the stalk map, so $\mathcal{E}_p = \TT^e$ and $\mathcal{F}_p = \TT^f$.

\begin{theorem}[Tropical Porteous Formula in Rank $0$] \label{main-thm}
Consider a morphism of tropical vector bundles $\phi: \mathcal{E} \to \mathcal{F}$ with bounded matrix entries over a rational polyhedral space $X$, where the bundles have ranks $e$ and $f$, respectively. Suppose $D_0(\phi)$ has the expected codimension $ef$. Then
$$[D_0^\phi (\mathcal{E}, \mathcal{F})] = \Delta_f^e \left( \frac{c_{[t]}(\mathcal{F})}{c_{[t]}(\mathcal{E})} \right),$$
where $c_{[t]}(-)$ denotes the Chern polynomial of the vector bundle.
\end{theorem}

\begin{proof}
By Proposition \ref{top-chern-degen}, the expected-codimension hypothesis gives $[D_0(\phi)] = c_{ef}(\mathcal{E}^\vee \otimes \mathcal{F}) \cdot [X]$. It remains to show that $c_{ef}(\EE^\vee \otimes \FF)$ equals the Sylvester determinant $\Delta_f^e(c_{[t]}(\FF)/c_{[t]}(\EE))$. By the splitting principle (Theorem \ref{splitting}), there exists a rational polyhedral space $Y$ and a morphism $f: Y \to X$ such that $f^*: A(X) \to A(Y)$ is an injective map. Na\"ively identifying $A(X)$ with its image in $A(Y)$, as we have that
$$c(\EE) = f^*(c(\EE)) = f^*\left(\sum_{i=0}^e c_i(\EE)\right) = \sum_{i=0}^e c_i(f^*\EE) = c(f^{*}(\EE)),$$
and $f^*\mathcal{E}$ splits into a direct sum of line bundles, say $f^*\EE = \bigoplus_{i=1}^e \mathcal{L}_i$, we may conclude from Whitney's formula that
$$c(\EE) = c(f^*(\EE)) = c\left(\bigoplus_{i=1}^e \mathcal{L}_i\right) = \prod_{i=1}^e (1 + c_1(\mathcal{L}_i)).$$
So, moving forward, with even greater na\"ivity, we write $\EE = \bigoplus_{i=1}^e \mathcal{L}_i$ (respectively $\FF = \bigoplus_{i=1}^f \MM_j$) with $\alpha_i = c_1(\LL_i)$ and $\beta_j = c_1(\MM_j)$. From this, it follows that
$$\EE^\vee \otimes \FF = \left( \bigoplus_{i=1}^e \LL_i^\vee \right) \otimes \left( \bigoplus_{j=1}^f \MM_j \right) = \bigoplus_{ij} (\LL_i^\vee \otimes \MM_j)$$
and since $c_1(\LL_i^\vee \otimes \MM_j) = -\alpha_i + \beta_j$ by Theorem \ref{c1-simple}, we finally have that
\begin{align*}
c(\EE^\vee \otimes \FF) &= c\left( \bigoplus_{ij} \LL_i^\vee \otimes \MM_j \right) &\text{by the splitting principle} \\
&= \prod_{ij} (1+\underbrace{(-\alpha_i + \beta_j)}_{\gamma_{ij}}) &\text{by Whitney's formula.}
\end{align*}
We then have that the top Chern class $c_{ef}(\EE^\vee \otimes \FF)$ is the highest degree term in the above expression, namely the product $\prod_{ij} \gamma_{ij}$. Lastly, we have that the Chern polynomials of $\EE$ and $\FF$ factor as
$$c_{[t]}(\EE) = \prod_{i=1}^e (1+ \alpha_i t) \quad \text{ and } \quad c_{[t]}(\FF) = \prod_{j=1}^f (1+\beta_j t).$$
So we can then conclude by Lemma \ref{sylvester} that
$$c_{ef}(\EE^\vee \otimes \FF) = \prod_{ij}(\beta_j - \alpha_i) = \Delta_f^e \left( \frac{c_{[t]} (\FF)}{ c_{[t]}(\EE) } \right),$$
as we had desired.
\end{proof}

\begin{example}
    Let us use the above formula for a simple instance of a computation. Consider a vector bundle morphism $\phi: \EE \to \OO_{\TP^1}$ from a rank $e$ bundle to the trivial line bundle over the tropical projective line $\TP^1$. By Proposition \ref{p1-split}, the bundle $\EE$ splits over $\TP^1$ into a direct sum $\bigoplus_{i=1}^e \mathcal{O}_{\TP^1}(d_i)$ for some tuple of integers $(d_1,\dots,d_e) \in \ZZ^e$. Carrying out Porteous formula, we have that
    $$[D_0^\phi(\EE,\OO_{\TP^1})] = \left[D_0^\phi\left( \bigoplus_{i=1}^e \mathcal{O}_{\TP^1}(d_i), \OO_{\TP^1} \right)\right] = \Delta_1^e \left( \frac{c_{[t]}(\OO_{\TP^1})}{c_{[t]} \left(  \bigoplus_{i=1}^e \mathcal{O}_{\TP^1}(d_i) \right)} \right).$$
    Expanding out the Chern polynomials, we have that this may be expressed as
    $$\Delta_1^e \left( \frac{c_0(\OO_{\TP^1}) + c_1(\OO_{\TP^1}) t}{\prod_{i=1}^e (1+c_1(\mathcal{O}_{\TP^1}(d_i))t)} \right)$$
    where the splitting into products downstairs is through Corollary \ref{whitney-cor}. We know that $c_0(\OO_{\TP^1}) = 1$, $c_1(\OO_{\TP^1}) = c_1(\OO_{\TP^1}(0)) = 0$, and $c_1(\mathcal{O}_{\TP^1}(d_i)) = d_i$ (Theorem \ref{serre-twist}). Using Lemma \ref{sylvester}, we may conclude that
    $$[D_0^\phi(\EE,\OO_{\TP^1})]= \Delta_1^e \left( \frac{1}{\prod_{i=1}^e (1+d_it)} \right) = (-1)^e d_1\cdots d_e.$$
\end{example}

\section{Future Directions}

We conclude by addressing two questions: what is needed to extend the rank-$0$ result to the full Porteous formula for arbitrary $D_k(\phi)$, and what motivates a tropical Porteous formula? The first leads to a discussion of the tropical Grassmannian and its role in reducing higher-rank degeneracy loci to the rank-$0$ case. The second returns to classical algebraic geometry and a conjecture that was the primary motivation for this work.

\subsection{Higher Rank Porteous}

In the classical setting, the full Porteous formula for $[D_k(\phi)]$ is deduced from the rank-$0$ case by a \textit{Grassmann bundle reduction}: one replaces the original morphism $\phi: \EE \to \FF$ with a new morphism on a larger space whose rank-$0$ locus maps onto $D_k(\phi)$. We outline how this strategy would proceed tropically, and identify the key ingredients that remain to be constructed.

\medskip

\noindent \textbf{The classical strategy.} Given $\phi: \EE \to \FF$ over $X$ with $\rank \EE = e$ and $\rank \FF = f$, and a fixed integer $0 \leq k < \min(e,f)$, the classical proof proceeds as follows:
\begin{enumerate}
    \item Form the Grassmann bundle $\rho: \Gr(e-k, \EE) \to X$, whose fiber over $x \in X$ parametrizes $(e-k)$-dimensional subspaces of $\EE_x$. Over $\Gr(e-k,\EE)$ there is a \textit{universal exact sequence}
    $$0 \to \mathcal{S} \to \rho^* \EE \to \mathcal{Q} \to 0,$$
    where $\mathcal{S}$ is the tautological subbundle of rank $e-k$, with $\mathcal{S}_{(x,K_x)} = K_x$ and $\mathcal{Q}_{(x,K_x)} = \EE_x / K_x$.
    \item Restrict $\rho^*\phi$ to $\mathcal{S}$ to obtain $\psi = \rho^*\phi|_{\mathcal{S}}: \mathcal{S} \to \rho^*\FF$, a morphism of ranks $e-k$ and $f$.
    \item Observe that $D_0(\psi) = \{(x,K_x) \mid K_x \subseteq \ker \phi_x\}$ maps onto $D_k(\phi)$ under $\rho$, since $\rank(\phi_x) \leq k$ if and only if $\ker \phi_x$ contains an $(e-k)$-plane.
    \item Apply the rank-$0$ formula: $[D_0(\psi)] = c_{(e-k)f}(\mathcal{S}^\vee \otimes \rho^*\FF) \cdot [\Gr(e-k,\EE)]$.
    \item Push forward via $\rho_*$ and compute using the Giambelli formula to obtain
    $$[D_k(\phi)] = \rho_*[D_0(\psi)] = \Delta_{f-k}^{e-k}\left(\frac{c_{[t]}(\FF)}{c_{[t]}(\EE)}\right).$$
\end{enumerate}
Note that the projectivization $\TP(\EE)$ used in the splitting principle is precisely the case $k = e-1$ of this construction, i.e.\ $\TP(\EE) = \Gr(1, \EE)$.

\medskip

\noindent \textbf{Tropical ingredients needed.} To carry out this program tropically, one would need:
\begin{enumerate}
    \item[\textbf{(I)}] A \textit{tropical Grassmann bundle} $\rho: \trop(\Gr(d, \EE)) \to X$ whose fiber over $x$ is $\trop(\Gr(d, \EE_x))$, realized as a rational polyhedral space. The Speyer--Sturmfels bijection between $\trop(G(d,n))$ and tropical $d$-planes \cite{tropgrass} gives candidates for the fibers, but assembling them into a bundle with the correct sedentary structure remains open.
    \item[\textbf{(II)}] A \textit{universal exact sequence} $0 \to \mathcal{S} \to \rho^*\EE \to \mathcal{Q} \to 0$ over $\trop(\Gr(e-k, \EE))$, with $\mathcal{S}_{(x,K_x)} = K_x$ the tautological subbundle of rank $e-k$.
    \item[\textbf{(III)}] A \textit{projection formula} for $\rho_*$, generalizing Lemma \ref{splitpf} from $\TP(\EE) = \Gr(1,\EE)$ to the higher Grassmann bundle.
    \item[\textbf{(IV)}] A \textit{Giambelli-type computation} verifying $\rho_*(c_{(e-k)f}(\mathcal{S}^\vee \otimes \rho^*\FF)) = \Delta_{f-k}^{e-k}(c_{[t]}(\FF)/c_{[t]}(\EE))$, which is formal given \textbf{(I)}--\textbf{(III)}.
\end{enumerate}
Of these, \textbf{(I)} is the essential obstacle: it requires globalizing the tropical Grassmannian into a fiber bundle compatible with the Mikhalkin--Rau framework.

\begin{question}
    Can one construct a tropical Grassmann bundle $\trop(\Gr(d,\EE)) \to X$ as a rational polyhedral space satisfying \textbf{(I)}--\textbf{(II)} above?
\end{question}

\medskip

\noindent \textbf{An alternative: iterated projectivization.} Since the projectivization $\TP(\EE) = \Gr(1,\EE)$ is already available, one could attempt to reduce $D_k$ to $D_0$ through iterated projectivization rather than a single Grassmann bundle construction. Concretely: over $\TP(\EE)$, the tautological sequence $0 \to \mathcal{O}_{\TP(\EE)}(-1) \to p^*\EE \to \mathcal{Q} \to 0$ and the composition $\mathcal{O}_{\TP(\EE)}(-1) \hookrightarrow p^*\EE \xrightarrow{p^*\phi} p^*\FF$ relate $D_{e-1}(\phi)$ to a zero locus upstairs. Iterating this on the successive quotients $\mathcal{Q}$ reduces $D_k$ to $D_0$ after $e-k$ steps, at the cost of working over a flag bundle rather than a Grassmann bundle. This avoids ingredient \textbf{(I)} entirely --- the splitting principle already constructs the relevant flag bundle --- but requires tracking iterated pushforwards and verifying the resulting expression simplifies to the Sylvester determinant.

\subsection{Towards the Brill--Noether Conjecture} Let $W^r_d \subseteq \Pic^d(\Gamma)$ denote the locus of divisor classes of \textit{degree} $d$ and \textit{rank} at least $r$ (see, e.g., \cite{CDPR, LPN}).

\begin{conjecture}[\cite{farbod}] \label{farbod}

Assume $\rho = g - (r+1)(g-d+r) \geq 0$. Then there exists a canonical tropical subvariety $Z^r_d \subseteq W^r_d$ of pure dimension $\rho$ such that 
\[
[Z^r_d] = \left(\prod_{i=0}^r \frac{i!}{(g-d+r+i)!}\right) [\Theta]^{g-\rho} \ .
\]
modulo tropical homological equivalence.
\end{conjecture}

This conjecture has been proven in the case $r=0$, where $W_d^0 = W_d$ is pure-dimensional (\cite{gross2022effective}) and $Z_d^0 = W_d$. This echoes the classical statement for subvarieties $W_d \subset \Pic^d(C)$ parameterizing effective divisor classes of degree $d$ on a curve $C$, i.e., images of the Abel--Jacobi maps $u_d: C_d \to \Jac(C) \cong \Pic^d(C)$.

\begin{question}
    In the classical setting, the Brill--Noether theorem is proven by expressing $W_d^r$ as a degeneracy locus of an evaluation map of vector bundles over $\Pic^d(C)$ and then applying Porteous' formula (see \cite{3264}). Could one approach this conjecture in an analogous way?
\end{question}

\printbibliography

\end{document}